\documentclass[11pt,leqno]{article}

\usepackage{amsthm,amsfonts,amssymb,amsmath,oldgerm}
\usepackage{fullpage}
\usepackage{graphicx}
\usepackage{mathrsfs}
\usepackage{xcolor}
\usepackage[colorinlistoftodos]{todonotes} %Use to show the notes

\numberwithin{equation}{section}

\renewcommand\d{\partial}
\renewcommand\a{\alpha}
\renewcommand\b{\beta}

\def\eps {\varepsilon}
\newcommand{\R}{\mathbb R}

\newcommand\LA{\left\langle}
\newcommand\RA{\right\rangle}

\newcommand{\RM}{{\mathbb{R}}}
\newcommand{\CM}{{\mathbb{C}}}
\newcommand{\NM}{{\mathbb{N}}}
\newcommand{\ZM}{{\mathbb{Z}}}
\newcommand{\vt}{\mathring{v}}
\newcommand{\ri}{\mathrm{i}}
\newcommand{\re}{\mathrm{e}}
\newcommand{\de}{\mathrm{d}}
\newtheorem{theorem}{Theorem}[section]
\newtheorem{proposition}[theorem]{Proposition}
\newtheorem{corollary}[theorem]{Corollary}
\newtheorem{lemma}[theorem]{Lemma}
\newtheorem{remark}[theorem]{Remark}
\theoremstyle{definition}
\newtheorem{definition}[theorem]{Definition}

\allowdisplaybreaks

\title{Nonlinear Subharmonic Dynamics of Spectrally Stable Lugiato-Lefever Periodic Waves}

\author{Mariana Haragus\thanks{FEMTO-ST Institute, University of Franche Comt\'e, 15b Avenue des Montboucons, 25030 Besan\c con cedex, France; \texttt{mharagus@univ-fcomte.fr}}
\quad Mathew A. Johnson\thanks{Department of Mathematics, University of Kansas, 1460 Jayhawk Boulevard, Lawrence, KS 66045, USA; \texttt{matjohn@ku.edu}}
\quad Wesley R. Perkins\thanks{Division of Mathematics and Computer Science, Lyon College, 2300 Highland Road, Batesville, AR 72501, USA; \texttt{wesley.perkins@lyon.edu}}
\quad Bj\"orn de Rijk\thanks{Karlsruhe Institute of Technology, Englerstra\ss e 2, 76131 Karlsruhe, Germany; \texttt{bjoern.rijk@kit.edu}}
}

\date{\today}

\begin{document}

\maketitle

%%%%%%%%%%%%%%%%%%%%%%%%%%%%%%%%%%%%%%%%%%%%%%%%%
%%%%%%%%%%%%%%%%%%%%%%%%%%%%%%%%%%%%%%%%%%%%%%%%%
\begin{abstract}
We study the nonlinear dynamics of perturbed, spectrally stable $T$-periodic stationary solutions of the Lugiato-Lefever equation (LLE), a damped nonlinear Schr\"odinger equation with forcing that arises in nonlinear optics. It is known that for each $N\in\NM$, such a $T$-periodic wave train is asymptotically stable against $NT$-periodic, i.e. subharmonic, perturbations, in the sense that initially nearby data will converge at an exponential rate to a (small) spatial translation of the underlying wave.  Unfortunately, in such results both the allowable size of initial perturbations as well as the exponential rates of decay depend on $N$ and, in fact, tend to zero as $N\to\infty$, leading to a lack of uniformity in the period of the perturbation. In recent work, the authors performed a delicate decomposition of the associated linearized solution operator and obtained linear estimates which are uniform in $N$. The dynamical description suggested by this uniform linear theory indicates that the corresponding nonlinear iteration can only be closed if one allows for a spatio-temporal phase modulation of the underlying wave. However, such a modulated perturbation is readily seen to satisfy a quasilinear equation, yielding an inherent loss of regularity. We regain regularity by transferring a nonlinear damping estimate, which has recently been obtained for the LLE in the case of localized perturbations to the case of subharmonic perturbations. Thus, we obtain a nonlinear, subharmonic stability result for periodic stationary solutions of the LLE that is uniform in $N$. This in turn yields an improved nonuniform subharmonic stability result  providing an $N$-independent ball of initial perturbations which eventually exhibit exponential decay at an $N$-dependent rate. Finally, we argue that our results connect in the limit $N \to \infty$ to previously established stability results against localized perturbations, thereby unifying existing theories.
\end{abstract}

\paragraph{Keywords:} Nonlinear Stability; Periodic Waves; Subharmonic Perturbations; Lugiato-Lefever Equation.

\paragraph{Subject Class:} 35B35; 35B10; 35Q60.

\section{Introduction} % \label{S:intro}

In this paper, we consider the asymptotic behavior and nonlinear stability against subharmonic perturbations of periodic stationary solutions of the Lugiato-Lefever equation (LLE)
\begin{equation}\label{e:LLE}
\psi_t = -\ri\b \psi_{xx} - (1+\ri\a)\psi + \ri|\psi|^2\psi + F,
\end{equation}
where $\psi(x,t)$ is a complex-valued function depending on a temporal variable $t \geq 0$ and a spatial variable $x \in \R$, the parameters $\alpha,\beta$ are real, and $F$ is a positive constant. The LLE was derived in 1987 from Maxwell's equations in~\cite{LL87} as a model to study pattern formation within the optical field in a dissipative and nonlinear cavity filled with a Kerr medium and subjected to a continuous laser pump.  In this context, $\psi(x,t)$ represents the field envelope, $F>0$ represents the normalized pump strength, $|\beta|=1$ is a dispersion parameter, and $\alpha>0$ represents a detuning parameter. Note that the case $\beta=1$, corresponding to a defocusing nonlinearity, is referred to as the ``normal" dispersion case, while the case $\beta=-1$, corresponding to a focusing nonlinearity, is referred to as the ``anomalous" dispersion case.  The LLE has recently become the subject of intense study in the physics literature, in part due to the fact that it has become a canonical model for high-frequency combs generated by microresonators in periodic optical waveguides; see, for example,~\cite{CGTM17} and references therein.

Several recent works have studied the existence of spatially periodic stationary solutions of~\eqref{e:LLE}, as well as the nonlinear dynamics about them.
Such solutions $\psi(x,t)=\phi(x)$ correspond to $T$-periodic solutions of the profile equation
\begin{equation}\label{e:profile}
-i\b \phi'' - (1+\ri\a)\phi + \ri|\phi|^2\phi + F=0.
\end{equation}
Smooth periodic solutions of~\eqref{e:profile} have been constructed using perturbative arguments, as well as local and global bifurcation theory~\cite{DH18_2,DH18_1,Go17,HSS19,MR17,MOT1}.
{It turns out that most of the constructed periodic waves are unstable under general bounded perturbations~\cite{DH18_2}. A class of periodic waves which are spectrally stable under general bounded perturbations has been identified in~\cite{DH18_1}. Nonlinear stability results have been obtained for co-periodic perturbations, i.e. $T$-periodic perturbations of the $T$-periodic wave $\phi$, in~\cite{MOT2,SS19} and for localized perturbations in the recent works~\cite{HJPR,ZUM22}. The results from~\cite{MOT2,SS19} can be extended to subharmonic perturbations, i.e. $NT$-periodic perturbations, provided spectral stability holds and the integer $N$ is fixed. It turns out that both the allowable size of initial perturbations as well as the exponential rates of decay, which depend on $N$, tend to zero as $N\to\infty$, leading to a lack of uniformity in the period of the perturbation. A linear stability result which holds uniformly in $N$ has been obtained in~\cite{HJP21}. The goal of the present work is to upgrade this result to the nonlinear level.}

\subsection{Spectral Stability Assumptions}

The local dynamics about a given $T$-periodic stationary solution $\phi$ of~\eqref{e:LLE} can be captured by considering the perturbed solution 
\begin{equation}\label{e:pert}
\psi(x,t)=\phi(x)+\tilde{v}(x,t)
\end{equation}
of~\eqref{e:LLE}, where {$\tilde v$ represents} some admissible perturbation.
Decomposing the solution $\phi=\phi_r+\ri\phi_i$ and the perturbation $\tilde{v}=\tilde{v}_r+\ri\tilde{v}_i$ into their real and imaginary parts\footnote{Going forward, we will slightly abuse notation and write our complex functions $f$ in the form $f = \left(\begin{smallmatrix}f_r \\ f_i\end{smallmatrix}\right)$.}, we see that $\eqref{e:pert}$ is a solution of~\eqref{e:LLE} provided that the real-valued functions $\tilde{v}_r$ and $\tilde{v}_i$ satisfy the system
\begin{equation}\label{e:lin}
\partial_t\left(\begin{array}{c}\tilde{v}_r\\\tilde{v}_i\end{array}\right)=\mathcal{A}[\phi]\left(\begin{array}{c}\tilde{v}_r\\\tilde{v}_i\end{array}\right)+\widetilde{\mathcal{N}}[\phi](\tilde{v}),
\end{equation}
where here $\mathcal A[\phi]$ is the (real) matrix differential operator
\begin{equation}\label{e:Aphi}
\mathcal A[\phi]=- I+\mathcal{J}\mathcal{L}[\phi],
\end{equation}
with
\[
\mathcal{J}=\left(\begin{array}{cc}0&-1\\1&0\end{array}\right),\quad
\mathcal{L}[\phi] = \left(\begin{array}{cc} -\b \d_x^2 - \a  + 3\phi_{r}^2 + \phi_{i}^2 & 2\phi_{r}\phi_{i} \\
  2\phi_{r}\phi_{i} & -\b \d_x^2 - \a  + \phi_{r}^2 + 3\phi_{i}^2\end{array}\right),
  \]
and where the nonlinearity is given by
\begin{align} \label{e:Nphi}
\widetilde{\mathcal{N}}[\phi]\left(\tilde{v}\right) = \mathcal{J}\left(\begin{array}{cc} 3\tilde{v}_{r}^2 + \tilde{v}_{i}^2 & 2\tilde{v}_{r}\tilde{v}_{i} \\
  2\tilde{v}_{r}\tilde{v}_{i} & \tilde{v}_{r}^2 + 3\tilde{v}_{i}^2\end{array}\right)\left(\begin{array}{c} \phi_r\\ \phi_i\end{array}\right) +
  \mathcal{J}\left|\left(\begin{array}{c}\tilde{v}_r\\\tilde{v}_i\end{array}\right)\right|^2\left(\begin{array}{c}\tilde{v}_r\\\tilde{v}_i\end{array}\right).
\end{align}

{The following notion of spectral stability served as the main hypothesis for the uniform linear stability result for subharmonic perturbations in~\cite{HJP21} as well as for the nonlinear stability results for localized perturbations in~\cite{HJPR, ZUM22}.}

\begin{definition}\label{Def:spec_stab}
Let $T>0$.   A smooth $T$-periodic stationary solution $\phi$ of~\eqref{e:LLE} is said to be \emph{diffusively spectrally stable} provided the following conditions hold:
\begin{enumerate}
\item the spectrum of the linear operator $\mathcal{A}[\phi]$, given by~\eqref{e:Aphi} and acting in $L^2(\R)$, satisfies 
\[
\sigma_{L^2(\R)}(\mathcal{A}[\phi])\subset\{\lambda\in\CM:\Re(\lambda)<0\}\cup\{0\};
\]
\item there exists $\theta>0$ such that for any $\xi\in[-\pi/T,\pi/T)$ the real part of the spectrum of the Bloch operator 
$\mathcal{A}_\xi[\phi]:=\mathcal{M}_\xi^{-1}\mathcal{A}[\phi]\mathcal{M}_\xi$, acting on $L^2_{\rm per}(0,T)$, satisfies
\[
\Re\left(\sigma_{L^2_{\rm per}(0,T)}(\mathcal{A}_\xi[\phi])\right)\leq-\theta \xi^2,
\]
where here $\mathcal{M}_\xi$ denotes the multiplication operator $\left(\mathcal{M}_\xi f\right)(x)=\re^{\ri\xi x}f(x)$;
\item $\lambda=0$ is a simple eigenvalue of the Bloch operator $\mathcal{A}_0[\phi]$, and the derivative $\phi'\in L^2_{\rm per}(0,T)$ of the periodic wave is an associated eigenfunction.
\end{enumerate}
\end{definition} 

Since the pioneering work of Schneider~\cite{S96,S98_1,S98_2}, the above notion of spectral stability (or extensions of it that account for more symmetries) has been standard in the linear and nonlinear stability analysis of periodic traveling or {steady} waves in dissipative systems. Indeed, it has been shown~\cite{RIJK22,DSSS,IYS99,JNRZ_13_1,JNRZ_13_2,SSSU} to imply important properties regarding the nonlinear dynamics against localized, or general bounded, perturbations, including the long-time dynamics under (large) phase modulations. We emphasize that the spectrally stable periodic steady waves of the LLE constructed in~\cite{DH18_1} are diffusively spectrally stable in the sense of Definition~\ref{Def:spec_stab}.

{Floquet-Bloch theory shows} that the spectrum of $\mathcal{A}[\phi]$ acting in $L^2(\R)$ is equal to the union of the spectra of the Bloch operators  $\mathcal{A}_\xi[\phi]$ acting in $L_{\rm per}^2(0,T)$ for  $\xi\in[-\pi/T,\pi/T)$. For subharmonic perturbations, the operator $\mathcal{A}[\phi]$ acts in $L_{\rm per}^2(0,NT)$, for some $N\in\mathbb{N}$, and its spectrum is the union of the spectra of the Bloch operators  $\mathcal{A}_\xi[\phi]$ acting in $L_{\rm per}^2(0,T)$ for $\xi$ in the discrete set $\{y \in [-\pi/T,\pi/T) : \re^{\ri y NT}=1\}$. {As a consequence,} diffusively spectrally stable periodic waves are necessarily spectrally stable to all subharmonic perturbations and, further, the spectrum possesses an $N$-dependent gap $\delta_N > 0$, i.e., we have
\begin{equation}\label{e:deltaN}
\Re\left(\sigma_{L^2_{\rm per}(0,NT)}(\mathcal{A}[\phi])\setminus\{0\}\right)\leq -\delta_N
\end{equation}
for each $N\in\mathbb{N}$. {We recall that the presence of the eigenvalue $0$ is a well-known consequence of the invariance of the LLE under spatial translations. Differentiation of the profile equation~\eqref{e:profile} with respect to $x$ shows that $\mathcal{A}[\phi](\phi') = 0$, and hence $\lambda=0$ is an eigenvalue with eigenfunction $\phi'$ of the operator $\mathcal{A}[\phi]$ acting in $L_{\rm per}^2(0,NT)$ for all $N \in \NM$.}

\subsection{Prior Subharmonic Stability Results}

The presence of the spectral gap~\eqref{e:deltaN} and the simplicity of the eigenvalue $\lambda=0$ allow us to obtain nonlinear stability against subharmonic perturbations for each arbitrary, but fixed, integer $N \in \NM$. This result follows as a straightforward extension of the nonlinear stability result for co-periodic perturbations in~\cite{SS19}.
  
\begin{theorem}[Subharmonic Nonlinear Asymptotic Stability]\label{T:SS}
  Let $\phi$ be a smooth $T$-periodic stationary solution of~\eqref{e:LLE} that is diffusively spectrally stable  in the sense of Definition~\ref{Def:spec_stab}. For each $N\in\NM$, take $\delta_N>0$ such that~\eqref{e:deltaN} holds.
  {Then, for every $\delta\in (0,\delta_N)$} there exist $\eps_\delta,C_\delta>0$ such that for each $v_0\in H^1_{\rm per}(0,NT)$ with $\|v_0\|_{H^1_{\rm per}(0,NT)}<\eps_\delta$ there exist a constant {$\gamma \in \R$} and a global mild solution $\psi \in C\big([0,\infty),H_{\rm per}^1(0,NT)\big)$ of~\eqref{e:LLE} with initial condition $\psi(0)=\phi+v_0$ satisfying
\[
|\gamma| \leq C_\delta\|v_0\|_{H^1_{\rm per}(0,NT)}, \qquad 
\left\|\psi(\cdot,t)-\phi(\cdot + \gamma)\right\|_{H_{\rm per}^1(0,NT)}\leq C_\delta\re^{-\delta t}\|v_0\|_{H_{\rm per}^1(0,NT)},
\]
for all $t \geq 0$.
\end{theorem}

While Theorem~\ref{T:SS} establishes nonlinear stability against $NT$-periodic perturbations for each fixed $N \in \NM$, it lacks uniformity in $N$ in two (related) aspects. Indeed, both the allowable size of initial perturbations $\eps_\delta$ and  the exponential rate of decay $\delta$ are controlled by the size of the spectral gap $\delta_N$.  Since $\delta_N\to 0$ as $N\to\infty$, it follows that both $\eps_\delta$ and $\delta$ chosen in Theorem~\ref{T:SS} necessarily tend to zero as $N\to\infty$, while the constant $C_\delta$ tends to infinity.  

{Addressing this issue requires} us to develop a strategy to uniformly handle the accumulation of $NT$-periodic eigenvalues near $\lambda=0$ as $N \to \infty$. In the proof of Theorem~\ref{T:SS}, the eigenvalue $\lambda=0$ was enclosed in a small ball $B(0,r_N)$, where the $N$-dependent radius $r_N > 0$ is chosen so that
\[
\sigma_{L^2_{\rm per}(0,NT)}\left(\mathcal{A}[\phi]\right)\cap B(0,r_N)=\{0\},
\]
leading to the associated spectral projection
\begin{equation} \label{e:proj}
\mathcal{P}_{0,N}=\frac{1}{2\pi \ri}\int_{\partial B(0,r_N)}\frac{\de z}{z-\mathcal{A}[\phi]} = \frac1N \phi' \LA\widetilde{\Phi}_0,\cdot \RA_{L^2_{\rm per}(0,NT)},
\end{equation}
onto the $1$-dimensional $NT$-periodic kernel of $\mathcal{A}[\phi]$ spanned by $\phi'$. Here, $\widetilde{\Phi}_0$ is the function spanning the $L^2_{\rm per}(0,T)$-kernel of the adjoint $\mathcal{A}[\phi]^*$ of the operator $\mathcal{A}[\phi]$,  normalized to satisfy %the $T$-periodic function spanning the kernel of the $L^2_{\rm per}(0,T)$-adjoint of $\mathcal{A}[\phi]$ satisfying 
$\smash{\langle\widetilde{\Phi}_0,\phi'\rangle_{L^2_{\rm per}(0,T)}=1}$. One then decomposes the semigroup, generated by $\mathcal{A}[\phi]$, as
\[
\re^{\mathcal{A}[\phi]t}={\mathcal{P}_{0,N}}+\re^{\mathcal{A}[\phi]t}\left(1-\mathcal{P}_{0,N}\right),
\]
and shows that for each $\delta\in(0,\delta_N)$ there exists a constant $C_\delta>0$ such that
\begin{equation*}%\label{e:SS_lin}
\left\|\re^{\mathcal{A}[\phi]t}\left(1-\mathcal{P}_{0,N}\right)f\right\|_{H^1_{\rm per}(0,NT)}\leq C_\delta\re^{-\delta t}\|f\|_{H^1_{\rm per}(0,NT)},
\end{equation*}
for all $t\geq 0$ and $f\in H^1_{\rm per}(0,NT)$. This allows one to establish the nonlinear stability result in Theorem~\ref{T:SS}; see~\cite{SS19} for more details.

The lack of uniformity thus stems from the fact that one must take the radius $r_N\to 0$ as $N\to\infty$. To establish uniform bounds one should instead work with a ball $B(0,r)$ with an $N$-independent radius $r>0$ and define the spectral projection
\[
\mathcal{P}:=\frac{1}{2\pi \ri}\int_{\partial B(0,r)}\frac{\de z}{z-\mathcal{A}[\phi]},
\]
associated with the generalized eigenspace corresponding to all the eigenvalues in the interior of the ball. Naturally, the dimension of this generalized eigenspace tends to infinity as $N\to\infty$, {a difficulty which must be handled} to establish uniform in $N$ decay estimates associated with the induced decomposition of the semigroup.  This analysis {of the semigroup} was carried out in detail in~\cite{HJP21}. We slightly reformulate the main result from~\cite{HJP21}; see Section~\ref{S:stab_loc}.
  
 \begin{theorem}[\cite{HJP21} Uniform Subharmonic Linear Asymptotic Stability]\label{T:sub_lin} Suppose $\phi$ is a smooth $T$-periodic stationary solution of~\eqref{e:LLE} that is diffusively spectrally stable  in the sense of Definition~\ref{Def:spec_stab}. Then, there exists a constant $C>0$ such that for every $N\in\NM$ and {$f\in L^2_{\rm per}(0,NT)$ there exist a constant $\sigma_\ell\in\R$ and a function $\gamma_\ell \in C\big([0,\infty),L^2_{\rm per}(0,NT)\big)$} with the following properties:
  \begin{align}  %\label{result1}
   \begin{split}
\left\|\re^{\mathcal{A}[\phi]t}f\right\|_{L^2_{\rm per}(0,NT)},\ |\sigma_\ell| &\leq C \|f\|_{L^1_{\rm per}(0,NT)\cap L^2_{\rm per}(0,NT)},
\\ 
\left\|\re^{\mathcal{A}[\phi]t}f - \frac1N\phi'(\cdot)\sigma_\ell\right\|_{L^2_{\rm per}(0,NT)} &\leq C(1+t)^{-\frac14}\|f\|_{L^1_{\rm per}(0,NT)\cap L^2_{\rm per}(0,NT)},
\\ 
\left\|\gamma_{\ell}(\cdot,t)- \frac1N\sigma_\ell\right\|_{L^2_{\rm per}(0,NT)} &\leq C(1+t)^{-\frac14}\|f\|_{L^1_{\rm per}(0,NT)\cap L^2_{\rm per}(0,NT)},
\\ 
\left\|\re^{\mathcal{A}[\phi]t}f-\phi'(\cdot)\gamma_\ell (\cdot,t)\right\|_{L^2_{\rm per}(0,NT)}
&\leq C (1+t)^{-\frac34}\|f\|_{L^1_{\rm per}(0,NT)\cap L^2_{\rm per}(0,NT)},
\end{split}\label{result3}
\end{align}
for all $t\geq0$.
\end{theorem}

{We point out that the use of the $L^1$-norm is essential in obtaining these uniform estimates as can be seen from the proof of Theorem~\ref{T:sub_lin} in~\cite{HJP21}. In fact, the nonuniform embedding $L^2_{\rm per}(0,NT)\hookrightarrow L^1_{\rm per}(0,NT)$ for which
  \[
\|f\|_{L^1_{\rm per}(0,NT)} \leq \sqrt{NT} \|f\|_{L^2_{\rm per}(0,NT)}, \qquad f \in L^2_{\rm per}(0,NT),
  \]
yields that the estimates~\eqref{result3} do hold in $L^2_{\rm per}(0,NT)$, but with an $N$-dependent constant $C\sqrt{NT}$. In addition, the uniform decay rates obtained in this result are exactly the ones found in the case of localized perturbations, the ideas of proof being also very similar, cf.~\cite{HJPR}. 
}

\subsection{Goal of Paper and Technical Challenge}

The aim of this paper is to upgrade Theorem~\ref{T:sub_lin} to the nonlinear level. We note that such a result is not at all straightforward. To gain insight into the main technical difficulties, we first provide an interpretation of the dynamics suggested by Theorem~\ref{T:sub_lin}. Suppose $\phi$ is a $T$-periodic stationary solution of~\eqref{e:LLE}, which is diffusively spectrally stable, and suppose that $\psi(x,t)$ is a solution of~\eqref{e:LLE} with initial data $\psi(x,0)=\phi(x)+\eps v_0(x)$ with $|\eps|\ll 1$ and  $v_0\in {L^2_{\rm per}(0,NT)}$. From~\eqref{result3}, it is natural to suspect that for $t\gg 1$ one should have
\begin{align*}
\psi(x,t)&\approx \phi(x) + \eps \re^{\mathcal{A}[\phi]t}v_0(x)\approx\phi(x) + \eps \phi'(x) \gamma_\ell(x,t)\approx\phi\left(x+\eps\gamma_\ell(x,t)\right),
\end{align*}
which is a small phase modulation of the background wave $\phi$. This suggests that, in order to capture the leading-order dynamics under perturbations (in a uniform way), one must incorporate into the nonlinear argument a phase modulation $\gamma_{\mathrm{nl}}(x,t)$ which depends on both space and time.\footnote{This stands in contrast to the classical approach of using a phase modulation that depends only on time.}

It turns out that the resulting \emph{inverse-modulated perturbation}
\begin{align*}
 v(x,t) = \psi\left(x - \gamma_{\mathrm{nl}}(x,t),t\right) - \phi(x),
\end{align*}
necessarily satisfies a quasilinear equation, thus yielding an inherent loss of regularity. In earlier work~\cite{JP21}, considering subharmonic perturbations of periodic waves in reaction-diffusion systems, this obstacle was overcome using a nonlinear damping estimate, which is an energy estimate effectively controlling higher Sobolev norms of the modulated perturbation in terms of its $L^2$-norm. However, in contrast to the reaction-diffusion case, such a damping estimate cannot be obtained with standard methods in our case due to low-order damping effects of the LLE. 

{The same difficulty appears in the study of the nonlinear stability of periodic waves for localized perturbations. Two alternative approaches to control regularity in the LLE were recently presented in~\cite{HJPR, ZUM22}. While the approach in~\cite{HJPR} relies on tame estimates on the \emph{unmodulated perturbation}
\[
 \tilde{v}(t) := \psi\left(t\right) - \phi.
\]
the approach in~\cite{ZUM22} uses nonlinear damping estimates on the \emph{forward-modulated perturbation}
\[
 \mathring{v}(x,t) = \psi\left(x,t\right) - \phi\left(x + \gamma_{\mathrm{nl}}(x,t)\right).
 \]
This second approach has the advantage that it requires less regularity on  initial data. We refer to~\cite{ZUM22} for a comparison of these methods; see also Section~\ref{S:nonlinear}.}

In this paper, we present an $N$-uniform nonlinear iteration scheme for subharmonic perturbations of Lugiato-Lefever periodic waves, loosely following the modulational approach of~\cite{JP21} for subharmonic perturbations of periodic reaction-diffusion waves, and, in order to control regularity, we transfer the method of~\cite{ZUM22} to the subharmonic setting in an $N$-uniform way.

\subsection{Main Result}

We state our main result, which establishes $N$-uniform nonlinear stability of diffusively spectrally stable $T$-periodic steady waves in the LLE against subharmonic perturbations and gives a precise modulational description of the local dynamics about the wave.

\begin{theorem}[Uniform Subharmonic Nonlinear Asymptotic Stability]\label{T:main}
Let $T > 0$ and suppose that $\phi$ is a smooth $T$-periodic stationary solution of~\eqref{e:LLE} that is diffusively spectrally stable  in the sense of Definition~\ref{Def:spec_stab}.\footnote{These hypotheses on $\phi$ are made throughout the whole paper.} Then, there exist constants $\eps, M > 0$ such that, for each $N\in\NM$, whenever {$v_0\in H^2_{\rm per}(0,NT)$} satisfies
\[
E_0:=\left\|v_0\right\|_{L^1_{\rm per}(0,NT) \cap H^2_{\rm per}(0,NT)}<\eps,
\]
there exist a constant $\sigma_{\mathrm{nl}}\in\R$, a modulation function
\[
\gamma_{\mathrm{nl}} \in C\big([0,\infty),H_{\rm per}^4(0,NT)\big) \cap C^1\big([0,\infty),H^2_{\rm per}(0,NT)\big), 
\]
and a global classical solution
\[
\psi \in C\big([0,\infty),H_{\rm per}^2(0,NT)\big) \cap C^1\big([0,\infty),L^2_{\rm per}(0,NT)\big),
\]
of~\eqref{e:LLE} with initial condition $\psi(0) = \phi + v_0$,
with the following properties:
\begin{align}
\begin{split}
\left\|\psi(\cdot,t) - \phi\right\|_{H^2_{\rm per}(0,NT)},\  |\sigma_{\mathrm{nl}}|
&\leq ME_0, \\
\left\|\psi(\cdot,t) - \phi\left(\cdot + \frac{1}{N}\sigma_{\mathrm{nl}} \right)\right\|_{H^2_{\rm per}(0,NT)}
&\leq ME_0 (1+t)^{-\frac{1}{4}}, \\
\left\|\gamma_{\mathrm{nl}}(\cdot,t)-\frac{1}{N}\sigma_{\mathrm{nl}} \right\|_{L^2_{\rm per}(0,NT)}
&\leq ME_0 (1+t)^{-\frac{1}{4}}, \\
\left\|\psi\left(\cdot,t\right) - \phi\left(\cdot + \gamma_{\mathrm{nl}}(\cdot,t) \right)\right\|_{H^2_{\rm per}(0,NT)}
&\leq ME_0 (1+t)^{-\frac{3}{4}}, \\
\left\|\partial_x \gamma_{\mathrm{nl}}(\cdot,t)\right\|_{H^3_{\rm per}(0,NT)},
\ \left\|\partial_t \gamma_{\mathrm{nl}}(\cdot,t)\right\|_{H^2_{\rm per}(0,NT)}
&\leq ME_0 (1+t)^{-\frac{3}{4}},\end{split} \label{e:MTmodder}
\end{align}
for all $t\geq 0$.
\end{theorem}

We note that the decay rates in Theorem~\ref{T:main} are sharp in the sense that they coincide with the (optimal) decay rates in the corresponding $N$-uniform linear result, Theorem~\ref{T:sub_lin}. {They also agree with the decay rates obtained for localized perturbations in~\cite{HJPR,ZUM22}.} Although the regularity on the initial perturbation required in Theorem~\ref{T:main} is higher than in the linear result, Theorem~\ref{T:sub_lin}, the regularity requirement in Theorem~\ref{T:main} is natural in the sense that it does reflect the amount of regularity needed to obtain a classical solution in $L^2_{\rm per}(0,NT)$ of the semilinear LLE~\eqref{e:LLE} via standard semigroup theory. Indeed, the domain of the linear operator $\beta \ri \partial_x^2$ acting on $L^2_{\rm per}(0,NT)$ is $H^2_{\rm per}(0,NT)$. We emphasize that the use of nonlinear damping estimates on the forward-modulated perturbation as in~\cite{ZUM22} allows us to preserve the regularity on the initial perturbation from the local existence theory, $v_0\in H^2_{\rm per}(0,NT)$, whereas the use of mild estimates on the unmodulated perturbation as in~\cite{HJPR} would require a higher regularity, $v_0\in H^6_{\rm per}(0,NT)$; see Section~\ref{S:nonlinear}.

\begin{remark}\label{R:SameConstant}
We point out that the spatial translates $\gamma$ and $\sigma_{\mathrm{nl}}/N$ in Theorems~\ref{T:SS} and~\ref{T:main} must be the same. To see this, fix $N\in\NM$, take $\eps_\delta$ as in Theorem~\ref{T:SS} and $\eps$ as in Theorem~\ref{T:main}. If $v_0\in H^2_{\rm per}(0,NT)$ satisfies $\smash{E_0 := \|v_0\|_{L^1_{\rm per}(0,NT) \cap H^2_{\rm per}(0,NT)}} < \min\{\eps,\eps_\delta\}$, then Theorems~\ref{T:SS} and~\ref{T:main} imply that the solution $\psi$ of~\eqref{e:LLE} with initial condition $\psi(0)=\phi+v_0$ satisfies
\[
\|\psi(\cdot,t)-\phi(\cdot+\gamma)\|_{H^1_{\rm per}(0,NT)}\leq C_\delta E_0e^{-\delta t}, \quad
\left\|\psi(\cdot,t) - \phi\left(\cdot + \frac{1}{N}\sigma_{\mathrm{nl}} \right)\right\|_{H^2_{\rm per}(0,NT)}\leq ME_0 (1+t)^{-\frac{1}{4}},
\]
for $t \geq 0$. So, the triangle inequality yields
\[
\left\|\phi(\cdot+\gamma) - \phi\left(\cdot+\frac{1}{N}\sigma_{\rm nl}\right)\right\|_{H^1_{\rm per}(0,NT)} \leq C_\delta E_0e^{-\delta t} +  ME_0 (1+t)^{-\frac{1}{4}},
\]
for all $t\geq0$, and taking $t\to\infty$ justifies that $\gamma=\frac{1}{N}\sigma_{\rm nl}$, as claimed.
\end{remark}

In light of this remark, we note that, as in~\cite{JP21}, the results of Theorems~\ref{T:SS} and~\ref{T:main} can be combined to yield an $N$-independent ball of initial perturbations which eventually exhibit exponential decay at an $N$-dependent rate to a translate of the periodic wave $\phi$. That is, we improve Theorem~\ref{T:SS} by showing that, for all but an at most finite number of $N\in\NM$, the allowable size of initial perturbations can be increased to a uniform size $\eps$.

\begin{corollary}\label{C:combined}
Let $\phi$, $\varepsilon$ and $M$ be as in Theorem~\ref{T:main}. Fix $N\in\NM$, let $\delta_N$ be as in~\eqref{e:deltaN}, and, for $\delta\in(0,\delta_N)$, let $\eps_\delta$ be as in Theorem~\ref{T:SS}. Then, there exist constants $T_\delta\geq 0$ and $M_\delta>0$ such that, whenever $v_0\in \smash{H^2_{\rm per}(0,NT)}$ satisfies $\smash{E_0 := \|v_0\|_{L^1_{\rm per}(0,NT) \cap H^2_{\rm per}(0,NT)}< \max\{\eps, \eps_\delta\}}$,\footnote{It is important to note that, since $\eps$ is independent of $N$ while $\eps_\delta\to0$ as $N\to\infty$, there is an at most finite number of positive integers $N$ for which we might have $\eps_\delta>\eps$.} there exist a constant $\sigma_{\rm nl} \in \R$ and a global mild solution $\psi \in C\big([0,\infty),H_{\rm per}^1(0,NT)\big)$ of~\eqref{e:LLE} with initial condition $\psi(0)=\phi+v_0$ satisfying
\[
\left\|\psi(\cdot,t) - \phi\left(\cdot + \frac{1}{N}\sigma_{\mathrm{nl}}\right)\right\|_{H^1_{\rm per}(0,NT)} \leq
\begin{cases}
ME_0(1+t)^{-\frac14}, & {\rm for}~ 0<t\leq T_\delta,\\
M_\delta E_0\re^{-\delta t}, & {\rm for}~ t>T_\delta.
\end{cases}
\]
\end{corollary}

Together with the (formal) observation that, as $N$ increases, functions in $\smash{H^m_{\rm per}(0,NT)}$ look more like functions in $H^m(\R)$, 
we see that  Corollary~\ref{C:combined} also serves to {\it formally} connect the subharmonic result in the limit as $N\to\infty$ to the localized result established in~\cite{HJPR, ZUM22}. 
In order to take the limit $N\to\infty$, we (only here) fix $E_0\in(0,\eps)$ independent of $N$ and choose a sequence $\smash{v_{0,N}\in L^1_{\rm per}(0,NT) \cap H^2_{\rm per}(0,NT)}$ such that $\smash{\|v_{0,N}\|_{L^1_{\rm per}(0,NT) \cap H^2_{\rm per}(0,NT)} = E_0}$ for each $N\in\NM$.\footnote{This argument essentially mimics choosing $w_0\in L^1(\R)\cap H^2(\R)$ such that $\widetilde{E}_0:=\|w_0\|_{L^1(\R)\cap H^2(\R)}\in(0,\eps)$ and a sequence $\smash{v_{0,N}\in L^1_{\rm per}(0,NT) \cap H^2_{\rm per}(0,NT)}$ such that $\smash{\|v_{0,N}\|_{L^1_{\rm per}(0,NT) \cap H^2_{\rm per}(0,NT)}} \to \widetilde{E}_0$ as $N\to\infty$.}
For $N$ sufficiently large, $T_\delta$ is defined to be the minimal time such that $ME_0(1+t)^{-\frac{1}{4}}\leq \eps_\delta$ for all $t\geq T_\delta$, see the proof in Section~\ref{S:proof} for motivation of the definition of $T_\delta$.   
Since $M$ and $E_0$ are independent of $N$ while $\eps_\delta \to 0$, we observe that $T_\delta \to\infty$ as $N\to\infty.$  Finally, noting that the spatial translate $\smash{\frac1N \sigma_{\rm nl}}$ converges to $0$ %\footnote{the fact that the spatial translate converges to $0$ corresponds to the fact that translation invariance is part of the point spectrum when studying subharmonic perturbations but part of the essential spectrum when studying localized perturbations.}
and the time $T_\delta$ diverges to $\infty$ as $N \to \infty$, we see that the localized result is indeed recovered when we, again formally, take $N\to\infty$ and fix $E_0$ independent of $N$.

\subsection{Outline of the Paper}

{In Section~\ref{S:stab_loc},} we collect and extend the relevant linear results obtained in~\cite{HJP21}. We decompose the semigroup $\smash{\re^{\mathcal{A}[\phi]t}}$ in a low- and high-frequency part, state associated $N$-uniform estimates on these parts and study the interaction of the low-frequency part with spatial and temporal derivatives. In Section~\ref{S:nonlinear}, we construct our nonlinear iteration scheme and establish an $N$-uniform nonlinear damping estimate to compensate for the loss of regularity exhibited by the scheme. In Section~\ref{S:proof}, we apply the linear estimates to our nonlinear iteration scheme and prove our main result, Theorem~\ref{T:main}, and Corollary~\ref{C:combined}. In Appendix~\ref{app:linear}, we provide a brief proof of the results presented in  Section~\ref{S:stab_loc}.  Finally, Appendix~\ref{app:local} is devoted to obtaining some local existence and regularity results necessary for our nonlinear analysis.

\subsection{Notation}

For convenience, for each  $N,m\in\NM$ we introduce the notations
\[
L^1_N := L^1_{\text{per}}(0,NT), \quad
L^2_N := L^2_{\text{per}}(0,NT), \quad
H^m_N:=H^m_{\text{per}}(0,NT).
\]
We equip $L^1_N$ and $L^2_N$ with the norm and the inner-product
\[
\|f\|_{L^1_N} = \int_0^{NT}|f(x)|dx \quad\mbox{and}\quad
\left\langle f,g\right\rangle_{L^2_N}=\int_0^{NT}f(x)g(x)dx,
\]
respectively, and equip $H^m_N$ with the norm
\[
 \|f\|_{H^m_N}^2 = \|f\|_{L^2_N}^2+\sum_{k=1}^m\big\|f^{(k)}\big\|_{L^2_N}^2,
\]
with $f^{(k)}$ representing the $k$-th order derivative of $f$. Notice that with these norms the Sobolev embedding $H^1_N\hookrightarrow L^\infty(\R)$, holds uniformly in $N$.\footnote{This can be seen by taking a constant $C > 0$, which bounds the norm of the continuous embedding $H^1(\R) \hookrightarrow L^\infty(\R)$, and a smooth, $T$-dependent, cut-off function $\omega \colon \R \to \R$ with 
$\|\omega\|_{L^\infty} = 1$, $\omega(x) = 1$ for $x \geq 0$ and $\omega(x) = 0$ for $x \leq -T$. For $N \in \NM$ define $\omega_N \colon \R \to \R$ by $\omega_N(x) = 1$ for $x \in [0,NT]$, $\omega_N(x) = \omega(x)$ for $x \leq 0$ and $\omega_N(x) = \omega(NT-x)$ for $x \geq NT$. It follows $\smash{\|f\|_{L^\infty} = \|f\omega_N\|_{L^\infty} \leq C\|f\omega_N\|_{H^1} \leq 3C\|\omega_N\|_{W^{1,\infty}}\|f\|_{H^1_N} = 3C\|\omega\|_{W^{1,\infty}}\|f\|_{H^1_N}}$ for $f \in H^1_N$, where we use the $NT$-periodicity of $f$.}

\paragraph{Acknowledgments:} MH was partially supported by French National Research Agency (ANR) through the project Optimal (grant number ANR-20-CE30-0004) and the EUR EIPHI program (grant number ANR-17-EURE-0002).
The work of MAJ was partially funded by the NSF under grant DMS-2108749, as well as the Simons Foundation Collaboration grant number 714021.
{The work of BdR was funded by the Deutsche Forschungsgemeinschaft (DFG, German Research Foundation) -- Project-ID 258734477 -- SFB 1173.}

\section{Subharmonic Linear Estimates}\label{S:stab_loc}

The starting point for our nonlinear analysis is the recent linear analysis in~\cite{HJP21}, where the authors established Theorem~\ref{T:sub_lin} by showing that the diffusive spectral stability assumptions {in Definition~\ref{Def:spec_stab} imply that, after splitting off the translational mode by applying the spectral projection $\mathcal{P}_{0,N}$, defined in~\eqref{e:proj}, the action of the semigroup $\re^{\mathcal{A}[\phi]t}$ can be decomposed into two components, one with $(1+t)^{-1/4}$-decay and one with $(1+t)^{-3/4}$-decay. 
The nonlinear analysis presented here requires the following slight extension of these results showing in addition that the component with slowest decay at rate $(1+t)^{-1/4}$ is smoothing.}

\begin{proposition}\label{P:lin}
Let $T>0$ and suppose that $\phi$ is a smooth, $T$-periodic stationary solution of~\eqref{e:LLE} that is diffusively spectrally stable {in the sense of Definition~\ref{Def:spec_stab}. Then, the action of the
linearized solution operator $\re^{\mathcal{A}[\phi]t}$ acting on $v\in L^2_N$ can be decomposed as 
\begin{align} \label{e:semidecomp0}
\re^{\mathcal{A}[\phi]t}v = {\mathcal{P}_{0,N}} v + \phi' s_{p,N}(t)v + {S}_N(t)v,
\end{align}
where $\mathcal{P}_{0,N}$ is the spectral projection~\eqref{e:proj}, and the operators $s_{p,N}(t)$ and ${S}_N(t)$ have the following properties.}
\begin{enumerate}
  \item  For all integers $j,l,k\geq 0$ there exists an $N$-independent constant $C_{j,l,k}>0$ such that
\begin{equation}\label{L:mod_bd1}
\begin{aligned}
  &\left\|\partial_x^l \partial_t^j s_{p,N}(t) \partial^k_x v\right\|_{L_N^2} \leq C_{j,l,k}(1+t)^{-\frac{1}{4}-\frac{l+j}{2}}\|v\|_{L_N^1} , & & v\in H^k_N, \\
  &\left\|\partial_x^l \partial_t^j s_{p,N}(t) \partial^k_x v\right\|_{L_N^2} \leq C_{j,l,k}(1+t)^{-\frac{l+j}{2}}\|v\|_{L_N^2}, & & v\in H^k_N,
\end{aligned}
\end{equation}
for all $t\geq 0$.
\item {There exists an $N$-independent constant $C>0$ such that
  \begin{equation} \label{L:mod_bd2}
\begin{aligned}
 &\left\|{S}_{N}(t)  v\right\|_{L_N^2}\leq C(1+t)^{-\frac{3}{4}}\|v\|_{L_N^1\cap L_N^2}, & & v\in L_N^2,
\end{aligned}
\end{equation}
for all $t\geq 0$.}
\end{enumerate}
\end{proposition}

{For $j=l=k=0$ this result has been proved in~\cite[Section 4]{HJP21}. We summarize the arguments leading to the additional estimates~\eqref{L:mod_bd1} in Appendix~\ref{app:linear}.
Notice that the linear stability result in Theorem~\ref{T:sub_lin} is an immediate consequence of the proposition above with the choice
\[
\sigma_\ell:=\LA\widetilde{\Phi}_0,f\RA_{L^2_N},\quad
\gamma_\ell(\cdot,t) := \frac1N\LA\widetilde{\Phi}_0,f\RA_{L^2_N} + s_{p,N}(t)f,
\]
for $f\in L_N^2$, where $\widetilde{\Phi}_0$ is the $T$-periodic function in the formula for the spectral projection~\eqref{e:proj}.

For our purposes, it is convenient to slightly modify the decomposition~\eqref{e:semidecomp0} such that the first two terms on the right hand side of~\eqref{e:semidecomp0}, which have slowest decay, vanish at $t=0$.\footnote{This will facilitate the choice of the modulation functions in~\S\ref{sec:invmodpert}.} To this end, we introduce a smooth cutoff function $\chi:[0,\infty)\to[0,1]$, which vanishes on $[0,1]$ and equals $1$ on $[2,\infty)$, and write
\begin{align} \label{e:semidecomp}
\re^{\mathcal{A}[\phi]t}v = {\chi(t)\mathcal{P}_{0,N}} v + \phi' \widetilde s_{p,N}(t)v + \widetilde{S}_N(t)v,
\end{align}
with    
    \[
    \widetilde s_{p,N}(t) = \chi(t) s_{p,N}(t), \quad
    \widetilde{S}_N(t) = (1-\chi(t))\left(\mathcal{P}_{0,N}+\phi' \widetilde s_{p,N}(t)\right) + {S}_{N}(t).
    \]
Because $\chi$ equals $1$ on $[2,\infty)$, the inequalities~\eqref{L:mod_bd1} and~\eqref{L:mod_bd2} hold for $\widetilde s_{p,N}(t)$ and $\widetilde{S}_N(t)$ as well.}

\section{Nonlinear Iteration Scheme}\label{S:nonlinear}

The goal of this section is to introduce the nonlinear iteration scheme that will be employed in the next section to prove our nonlinear stability result, Theorem~\ref{T:main}. To this end, let $\phi$ be a smooth, $T$-periodic stationary solution of the LLE~\eqref{e:LLE}, which is diffusively spectrally stable in the sense of Definition~\ref{Def:spec_stab}. Fix $N \in \NM$ and consider the perturbed solution $\psi(t)$ of~\eqref{e:LLE} with initial condition 
\[
\psi(0) = \phi + v_0,
\]
where $v_0 \in H^2_N$ is sufficiently small. Noting that the linear operator $\beta \ri \partial_{x}^2$ acting on the space $L^2_N$ with domain $H^2_N$ generates a $C_0$-semigroup, and the mapping $\psi \mapsto -(1+\ri\a)\psi + \ri|\psi|^2\psi + F$ is locally Lipschitz continuous on $H^2_N$, standard semigroup theory readily yields local existence and uniqueness of the perturbed solution $\psi(t)$; see, for instance,~\cite[Proposition~6.1.7]{Pazy}.

\begin{proposition}[Local Theory for the Perturbed Solution] \label{p:local_unmod}
For any $v_0 \in H^2_N$, there exists a maximal time $T_{\max} \in (0,\infty]$ such that~\eqref{e:LLE} admits a unique classical solution
\begin{align*} \psi \in C\big([0,T_{\max}),H^2_N\big) \cap C^1\big([0,T_{\max}),L^2_N\big), \end{align*}
with initial condition $\psi(0) = \phi + v_0$. In addition, if $T_{\max} < \infty$, then
\begin{align} \limsup_{t \uparrow T_{\max}} \left\|\psi(t)\right\|_{H^2_N} = \infty. \label{e:psiblowup}
\end{align}
\end{proposition}

It is well-known that direct control on the \emph{unmodulated perturbation}
\[
 \tilde{v}(t) := \psi\left(t\right) - \phi,
\]
which satisfies the semilinear equation~\eqref{e:lin}, is not strong enough to close a nonlinear stability argument. Indeed, {iterative estimates on the Duhamel formula associated with~\eqref{e:lin} are too weak to close the nonlinear argument because of the presence of the constant, non-decaying, term $\mathcal P_{0,N}v$ in  the decomposition~\eqref{e:semidecomp0} of the semigroup $\re^{\mathcal A[\phi] t}$. To overcome this lack of decay, a standard approach} is to consider the \emph{temporally-modulated perturbation}
\[
 \breve{v}(x,t) := \psi(x,t) - \phi\left(x+\sigma(t)\right),
\]
{which then leads to the result in Theorem~\ref{T:SS}.
However, this result is not uniform in $N$ and the $N$-uniform decay rate $(1+t)^{-1/4}$ of the remaining terms in the semigroup $\smash{\re^{\mathcal A[\phi]t}}$ is too weak to close the nonlinear iteration in the presence of a quadratic nonlinearity.} To overcome this obstacle, we introduce the \emph{inverse-modulated perturbation}
\begin{align} \label{e:definvmod}
 v(x,t) := \psi\left(x - \gamma(x,t) - \frac{1}{N} \sigma(t),t\right) - \phi(x),
\end{align}
where the temporal modulation function $\sigma(t)$ is  {chosen to account for the nondecaying term $\mathcal P_{0,N}v$} of the semigroup, whereas the spatio-temporal phase modulation $\gamma(x,t)$ is chosen to account for the  {term with slowest algebraic decay rate $(1+t)^{-1/4}$ in the semigroup decomposition~\eqref{e:semidecomp}.}
The idea of a spatio-temporal phase modulation to capture the most critical diffusive dynamics stems from the nonlinear stability analysis of periodic waves in reaction-diffusion systems against localized and nonlocalized perturbations; see~\cite{DSSS,RIJK22,JZ_11_1,JNRZ_13_1,JNRZ_13_2,SSSU}. The approach was then later adapted to obtain $N$-uniform results in the case of subharmonic perturbations in~\cite{JP21} by including an additional nondecaying temporal modulation $\sigma(t)$. Note that this methodology also
extends to systems with additional conservation laws, thus allowing for additional modulation functions~\cite{JP22_1}.

An issue is that the inverse-modulated perturbation {$v$ defined in~\eqref{e:definvmod}} satisfies a quasilinear equation, yielding an apparent loss of derivatives in the nonlinear iteration scheme. To regain regularity one often relies on nonlinear damping estimates, which are energy estimates, effectively providing control of higher Sobolev norms of the inverse-modulated perturbation in terms of its $L^2$-norm. If the underlying equation is parabolic, such nonlinear damping estimates can be obtained {from smoothing properties of the (analytic) semigroup}; see for instance~\cite{JZ_11_1} and~\cite{JP21} for the case of localized and subharmonic perturbations in reaction-diffusion systems, respectively. In general, however, the existence of nonlinear damping estimates is not guaranteed and their derivation can be tedious and lengthy; see, for instance, the delicate analyses~\cite[Appendix~A]{JZN},~\cite[Section~5]{MS04} and~\cite{RZ16} in the case of hyperbolic-parabolic systems.

In cases where nonlinear damping estimates are unavailable or difficult to obtain, there is an alternative approach to control regularity, which was introduced in~\cite{RS18}. The key idea {in~\cite{RS18}} is to incorporate tame estimates on the unmodulated perturbation $\tilde{v}(t)$, which  satisfies a semilinear equation in which no derivatives are lost. This approach was applied in the stability analysis~\cite{HJPR} of periodic waves in the LLE~\eqref{e:LLE} against localized perturbations and works as long as the underlying equation is semilinear. Moreover, it has the advantage that it does not rely on localization or periodicity properties of perturbations and, thus, can be applied in case of pure $L^\infty$-perturbations, cf.~\cite{RS18,RIJK22}.

Recently, a nonlinear damping estimate was established for the LLE in~\cite{ZUM22} in the case of localized perturbations. The approach in~\cite{ZUM22} is to first derive a damping estimate for the \emph{forward-modulated perturbation}
\begin{align} \label{e:defforpert}
 \mathring{v}(x,t) := \psi\left(x,t\right) - \phi\left(x + \gamma(x,t) + \frac{1}{N} \sigma(t)\right),
\end{align}
and then exploit its equivalence to the inverse-modulated equation, modulo absorbable errors. Here, the modulation functions $\gamma$ and $\sigma$ are precisely those chosen from the inverse-modulated perturbation in~\eqref{e:definvmod}.
The advantage of using a nonlinear damping estimate over the approach in~\cite{HJPR} is that it requires less regularity on the initial data, as can be seen by comparing~\cite[Theorem~6.2]{ZUM22} with~\cite[Theorem~1.3]{HJPR}, see also Remark~\ref{R:Choice}. In addition, as pointed out in~\cite{ZUM22}, it allows for sharp bounds in case of a nonlocalized initial phase {modulation, and has  the possibility (to be checked in individual cases) of extension to quasilinear equations. We refer to~\cite{ZUM22} for further discussion and comparison with the above method from~\cite{HJPR}.}

Thus, motivated by the possibility to allow for less regular initial data, we choose to control regularity in this work by transferring the nonlinear damping estimate in~\cite{ZUM22} to the case of subharmonic perturbations. 

The rest of this section is structured as follows. First, we derive the (quasilinear) equation for the inverse-modulated perturbation $v$, and obtain $N$-uniform estimates on the nonlinearities. Next, we show that the critical terms in the Duhamel formula of the inverse-modulated perturbation can be compensated for by making a judicious choice for the phase modulation functions $\sigma(t)$ and $\gamma(x,t)$. Subsequently, we establish local well-posedness of the integral system consisting of $v$, $\sigma$ and~$\gamma$. Finally, in an effort to control regularity in this system, we consider the forward-modulated perturbation and derive a suitable $N$-uniform nonlinear damping estimate. {The result in Theorem~\ref{T:main} is proved in Section~\ref{S:proof} with $\gamma_{\mathrm{nl}} = \gamma +\sigma/N$ and a suitably chosen constant $\sigma_{\mathrm{nl}}$.}

\begin{remark}\label{R:Choice}
  We note that it is possible to adapt the method of~\cite{RS18, HJPR} to the current setting in order to regain regularity and obtain a nonlinear subharmonic stability result that is uniform in $N$.  This adaptation is complicated, however, by the fact that the temporally-modulated perturbation, $\breve{v}$, and the inverse-modulated perturbation, $v$, want to naturally select different temporal modulations, $\sigma(t)$, at least under the approach of~\cite{RS18, HJPR}. Despite this complication, we were able to establish a result\footnote{We do not report the full result or details here.} similar to Theorem~\ref{T:main}, albeit with the requirement that {$v_0\in H^6_N$, instead of $v_0\in H^2_N$}, where the higher regularity is necessary to guarantee optimal decay results in the absence of a nonlinear damping estimate.
\end{remark}

\subsection{The Inverse-Modulated Perturbation} \label{sec:invmodpert}

Applying the differential operator $\partial_t - \mathcal A[\phi]$ to {the formula~\eqref{e:definvmod} for the inverse-modulated perturbation $v$ while using that $\psi(t)$ and $\phi$ solve the Lugiato-Lefever equation~\eqref{e:LLE}, we obtain} the quasilinear equation
\begin{equation} \label{e:eqinvmod}
\left(\partial_t-\mathcal{A}[\phi]\right)\left(v + \gamma\phi' + \frac{1}{N}\sigma\phi'\right) = \mathcal{N}(v,\gamma,\partial_t \gamma,\partial_t \sigma) + (\d_t - \mathcal{A}[\phi])(\gamma_x v),
\end{equation}
where $\mathcal{A}[\phi]$ is the linear operator defined by~\eqref{e:Aphi} and the nonlinearity is given by
\begin{align*}
\mathcal{N}(v,\gamma,\gamma_t,\sigma_t) = \mathcal{Q}(v,\gamma) + \d_x \mathcal{R}(v,\gamma,\gamma_t,\sigma_t) + \d_x^2 \mathcal{P}(v,\gamma), 
\end{align*}
where
\[
\mathcal{Q}(v,\gamma) =
(1-\gamma_x)\mathcal{J}\left[\left(\begin{array}{cc} 3v_r^2+v_i^2 & 2v_r v_i\\ 2v_r v_i & v_r^2+3v_i^2\end{array}\right)\phi+|v|^2v\right],
\]
is (at least) quadratic in $v$ and where
\begin{align*}
\mathcal{R}(v,\gamma,\gamma_t,\sigma_t) &= -\gamma_t v -\frac{1}{N}\sigma_t v + \beta\mathcal{J}\left[\frac{\gamma_{xx}}{\left(1-\gamma_x\right)^2} v - \frac{\gamma_x^2}{1-\gamma_x}\phi'\right] ,\\
\mathcal{P}(v,\gamma) &= -\beta\mathcal{J}\left[\gamma_x + \frac{\gamma_x}{1-\gamma_x}\right] v,
\end{align*}
contain all terms which are linear in $v$.

Using the $N$-uniform embedding $H^1_N\hookrightarrow L^\infty(\R)$,
the following estimate on the nonlinearity is straightforward to verify.

\begin{lemma} \label{lem:mod_nonlL2}
Fix an $N$-independent constant $K > 0$. There exists an $N$-independent constant $C > 0$ such that the inequality
\begin{align*}
   \left\|\mathcal{N}(v,\gamma,\gamma_t,\sigma_t)\right\|_{L_N^1\cap L_N^2} &\leq C\left( \left\|v\right\|_{L_N^2} \left\|v\right\|_{H_N^1} + \left\|(\gamma_x,\gamma_t)\right\|_{H_N^2 \times H_N^1} \left(\left\|v\right\|_{H_N^2} + \left\|\gamma_x\right\|_{L_N^2}\right) + |\sigma_t|\|v\|_{H_N^1}\right),
\end{align*}
holds for all $v \in H_N^2$, $(\gamma,\gamma_t) \in H_N^3 \times H^1_N$ and $(\sigma,\sigma_t) \in \R \times \R$ satisfying $\|v\|_{L^\infty}, \|\gamma_{xx}\|_{L^{\infty}} \leq K$ and $\|\gamma_x\|_{L^\infty} \leq \frac{1}{2}$.
\end{lemma}

Next, we introduce the modulation functions $\sigma$ and $\gamma$.
 The decomposition~\eqref{e:semidecomp} of the semigroup $\re^{\mathcal{A}[\phi]t}$ in which the first two terms, with lower decay, vanish at $t=0$ allows us to consider modulation functions which vanish identically at $t = 0$, i.e. such that  $\sigma(0)=0$ and $\gamma(\cdot,0)=0$. Then, the Duhamel formulation associated with~\eqref{e:eqinvmod} reads
\begin{equation}\label{e:duhamel1}
v(t) + \frac{1}{N}\sigma(t)\phi' + \gamma(t)\phi' = \re^{\mathcal{A}[\phi]t}v_0 + \int_0^t \re^{\mathcal{A}[\phi](t-s)}\mathcal{N}(v,\gamma,\partial_s \gamma,\partial_s \sigma)(s) \de s + \gamma_x(t)v(t).
\end{equation}
Together with the semigroup decomposition~\eqref{e:semidecomp}, and the formula~\eqref{e:proj} for the spectral projection $\mathcal{P}_{0,N}$, this recommends the (implicit) choices
\begin{align}
    \sigma(t) &= \chi(t)\LA\widetilde{\Phi}_0,v_0\RA_{L^2_N} + \int_0^t \chi(t-s)\LA\widetilde{\Phi}_0,\mathcal{N}(v,\gamma,\partial_s \gamma,\partial_s \sigma)(s)\RA_{L^2_N} \de s, \label{e:intsigma}\\
    \gamma(t) &= \widetilde s_{p,N}(t)v_0 + \int_0^t \widetilde s_{p,N}(t-s)\mathcal{N}(v,\gamma,\partial_s \gamma,\partial_s \sigma)(s)\de s, \label{e:intgamma}
\end{align}
so that $\sigma(t)$ accounts for the non-decaying $\chi(t)\mathcal{P}_{0,N}$-terms and $\gamma(x,t)$ for the slowly decaying $\widetilde s_{p,N}(t)$-terms on the right-hand side of~\eqref{e:duhamel1}. 
{Subtracting~\eqref{e:intsigma} and~\eqref{e:intgamma}
from~\eqref{e:duhamel1} yields the equation for the inverse-modulated perturbation}
\begin{equation}
    v(t) = \widetilde{S}_N(t)v_0 + \int_0^t \widetilde{S}_N(t-s)\mathcal{N}(v,\gamma,\partial_s \gamma,\partial_s \sigma)(s)\de s + \gamma_x(t)v(t). \label{e:intv2}
\end{equation}

Recalling the definition~\eqref{e:definvmod} of the inverse-modulated perturbation, one observes that~\eqref{e:intsigma}-\eqref{e:intgamma} forms a closed integral system in terms of $\sigma$ and $\gamma$. A standard contraction mapping argument then yields the  local existence and uniqueness result for the modulation functions.

\begin{proposition}[Local Theory for the Phase Modulations] \label{p:gamma}
Taking $\psi$ and $T_{\max}$ as in Proposition~\ref{p:local_unmod}, there exist a maximal time $\tau_{\max} \in (0,T_{\max}]$ and an $N$-independent constant $r_0 > 0$ such that the integral system~\eqref{e:intsigma}-\eqref{e:intgamma}, with $v$ given by~\eqref{e:definvmod}, has a unique solution
\begin{align*} (\sigma,\gamma) \in C\big([0,\tau_{\max}),\R \times H^4_N\big) \cap C^1\big([0,\tau_{\max}),\R \times H^2_N\big), \end{align*}
with $(\sigma,\gamma)(0)=0$ satisfying
\begin{align*}
\left\|\left(\sigma(t),\partial_t \sigma(t),\gamma(t),\partial_t \gamma(t)\right)\right\|_{\R \times \R \times H^4_N \times H^2_N} < r_0, \qquad \|\gamma_x(t)\|_{L^\infty} \leq \frac12, 
\end{align*}
for all $t \in [0,\tau_{\max})$. Finally, if $\tau_{\max} < T_{\max}$, then
\begin{align} \limsup_{t \uparrow \tau_{\max}} \left\|\left(\sigma(t),\partial_t \sigma(t),\gamma(t),\partial_t \gamma(t)\right)\right\|_{\R \times \R \times H^4_N \times H^2_N} \geq r_0. \label{e:gammablowup}
\end{align}
\end{proposition}

We prove this proposition in Appendix~\ref{app:local}. Then, for the inverse-modulated perturbation $v$ we obtain the following local existence result.

\begin{proposition}[Local Theory for The Inverse-Modulated Perturbation] \label{C:local_v}
Taking $\psi$ and $T_{\max}$ as in Proposition~\ref{p:local_unmod} and $\sigma,\gamma$ and $\tau_{\max}$ as in Proposition~\ref{p:gamma}, the inverse-modulated perturbation $v$, defined by~\eqref{e:definvmod}, satisfies $v \in C\big([0,\tau_{\max}),L^2_N\big)$. Moreover, for any $t \in [0,\tau_{\max})$ it holds $v(t) \in H^2_N$.\footnote{We note that it is not clear that the inverse-modulated perturbation $v \colon [0,\tau_{\max}) \to H^2_N$ is continuous. A standard approach to prove continuity of $v$ would be to apply the mean value theorem to the perturbed solution $\psi(t)$ and its derivatives. This would however require boundedness of the third derivative of $\psi(t)$, which does not follow from Proposition~\ref{p:local_unmod}.} 
\end{proposition}
\begin{proof}
First, notice that $v = V \circ F$ where $V$ is the continous mapping defined in Lemma~\ref{L:loc_gamma} and $F \colon [0,\tau_{\max}) \to L^2_N \times \R \times [0,T_{\max})$ is defined by $F(t) = (\gamma(t),\sigma(t),t)$. 
By Proposition~\ref{p:gamma} the map $F$ is continuous which together with the continuity of $V$ implies that $v \in C\big([0,\tau_{\max}),L^2_N\big)$. 

Next, let $t\in [0,\tau_{\max})$. By Proposition~\ref{p:gamma} it holds $\|\gamma_x(t)\|_{L^\infty} \leq \frac12$. So, the map $A_t \colon \R \to \R$ given by
$$ A_t(x) = x - \gamma(x,t) - \frac{1}{N} \sigma(t),$$
is invertible. Moreover, the $NT$-periodicity of $\gamma(\cdot,t)$ implies $A_t(NT) - A_t(0) = NT$. Hence, using Young's inequality, the $N$-uniform embedding $H^1_N \hookrightarrow L^\infty(\R)$ and the substitution $y = A_t(x)$, we establish a constant $C > 0$ such that for any $f \in H^2_N$ it holds
\begin{align}
\label{e:subst2}
\begin{split}
\left\|f\left(A_t(\cdot)\right)\right\|_{H^2_N}^2 
&\leq \int_{0}^{NT} \left|f\left(A_t(y)\right)\right|^2 \de y
+ 2\int_{0}^{NT} \left|f''\left(A_t(y)\right)\right|^2 A_t'(y)^4 \de y \\ 
&\qquad\quad + \,\int_{0}^{NT} \left|f'\left(A_t(y)\right)\right|^2 \left(A_t'(y)^2 + 2A_t''(y)^2\right) \de y \\
&= \int_{A_t(0)}^{A_t(NT)} \frac{\left|f\left(x\right)\right|^2}{1-\gamma_x\left(A_t^{-1}(x),t\right)} \de x\\ 
&\qquad\quad
+ 2\int_{A_t(0)}^{A_t(NT)} \left|f''\left(x\right)\right|^2\left(1-\gamma_x\left(A_t^{-1}(x),t\right)\right)^3 \de x\\ 
&\qquad\quad + \,\int_{A_t(0)}^{A_t(NT)} \left|f'\left(x\right)\right|^2\frac{\left(1-\gamma_x\left(A_t^{-1}(x),t\right)\right)^2 + 2\gamma_{xx}\left(A_t^{-1}(x),t\right)^2}{1-\gamma_x\left(A_t^{-1}(x),t\right)} \de x\\
&\leq C\|f\|_{H^2_N}^2\left(1+\|\gamma_{xx}(t)\|_{H^1_N}^2\right).
\end{split}
\end{align}
Therefore, we find $v(t) = \psi(A_t(\cdot),t) - \phi \in H^2_N$ by Propositions~\ref{p:local_unmod} and~\ref{p:gamma}. 
\end{proof}

We note that the primary success in the above nonlinear decomposition is that the only component of the linearized evolution $\re^{\mathcal{A}[\phi]t}$ that remains in the equation~\eqref{e:intv2} for $v(t)$ is the $\smash{\widetilde{S}_N(t)}$-component, which exhibits temporal decay at the rate $\smash{(1+t)^{-\frac34}}$, which is strictly faster than the (diffusive) decay rate $\smash{(1+t)^{-\frac14}}$ associated with the projected semigroup $\re^{\mathcal{A}[\phi]t}\left(1-\mathcal{P}_{0,N}\right)$.  Further, the nonlinear residual $\mathcal{N}$ depends only on derivatives of $\gamma$ and $\sigma$ which, recalling that $\smash{\widetilde s_{p,N}(0)=0}$ and $\chi(0) = 0$, satisfy
\begin{align} \label{e:intsigma2}
\partial_t^j \sigma(t) &= \partial_t^j \chi(t)\LA\widetilde{\Phi}_0,v_0\RA_{L^2_N} + \int_0^t \partial_t^j \chi(t-s)\LA\widetilde{\Phi}_0,\mathcal{N}(v,\gamma,\partial_s \gamma,\partial_s \sigma)(s)\RA_{L^2_N} \de s \\
\label{e:intgamma2}
\partial_x^\ell\partial_t^j{\gamma(x,t)} &= \partial_x^\ell\partial_t^j\widetilde s_{p,N}(t)v_0 +\int_0^t\partial_x^\ell\partial_t^j \widetilde s_{p,N}(t-s)\mathcal{N}(v,\gamma,\partial_s\gamma,\partial_s \sigma)(s)\de s
\end{align}
for all $\ell,j\in\NM_0$ and $t \in [0,\tau_{\max})$. We observe that the derivative $\chi'(t)$ vanishes for $t \geq 2$, whereas the temporal decay of $\widetilde s_{p,N}(t)$ improves to $\smash{(1+t)^{-\frac{3}{4}}}$ upon taking derivatives, cf.~Proposition~\ref{P:lin}. This suggests that the linear decay in an iteration scheme consisting of $v(t)$ and \emph{derivatives} of $\sigma(t)$ and $\gamma(t)$ is strong enough to close a nonlinear argument. Yet, the apparent loss of regularity needs to be addressed, which will be the purpose of the remainder of this section.

\subsection{The Forward-Modulated Perturbation}

First, the local existence and uniqueness of the forward-modulated perturbation, {$\vt$} defined by~\eqref{e:defforpert}, readily follows from Propositions~\ref{p:local_unmod} and~\ref{p:gamma}.

\begin{corollary}[Local Theory for The Forward-Modulated Perturbation] \label{C:local_vf}
Taking $\psi$ as in Proposition~\ref{p:local_unmod} and $\sigma,\gamma$ and $\tau_{\max}$ as in Proposition~\ref{p:gamma}, the forward-modulated perturbation $\vt$ defined by~\eqref{e:defforpert} satisfies $\vt \in C\big([0,\tau_{\max}),H^2_N\big) \cap C^1\big([0,\tau_{\max}),L_2^N\big)$.
\end{corollary}
\begin{proof}
Let $t,s \in [0,\tau_{\max})$. With the aid of the mean value theorem and the embedding $H^1_N \hookrightarrow L^\infty(\R)$ we establish a constant $C > 0$ such that
\begin{align*}
\|\vt(t) - \vt(s)\|_{H^2_N} &\leq C\|\phi'\|_{W^{2,\infty}} \left(1 + \|\gamma_x(t)\|_{H^2_N} + \|\gamma_x(s)\|_{H^2_N}\right)^2\left(\|\gamma(t) - \gamma(s)\|_{H^2_N} + |\sigma(t) - \sigma(s)|\right) \\
&\qquad + \, \|\psi(t) - \psi(s)\|_{H^2_N},
\end{align*}
and
\begin{align*}
\|\partial_t \vt(t) - \partial_s \vt(s)\|_{L^2_N} &\leq \|\phi'\|_{L^\infty} \left(\|\partial_t \gamma(t) - \partial_s \gamma(s)\|_{L^2_N} + |\sigma'(t) - \sigma'(s)|\right) + \|\partial_t \psi(t) - \partial_s \psi(s)\|_{L^2_N}\\
& \qquad + \, \|\phi''\|_{L^\infty} \left(\|\partial_t \gamma(t)\|_{H^1_N} +  |\sigma'(t)|\right)\left(\|\gamma(t) - \gamma(s)\|_{L^2_N} + |\sigma(t) - \sigma(s)|\right),
\end{align*}
which yields the proof by invoking Propositions~\ref{p:local_unmod} and~\ref{p:gamma}.
\end{proof}

Applying the operator $\partial_t - \mathcal A[\phi]$ to~\eqref{e:defforpert} while using that $\phi$ and $\psi(t)$ are solutions of~\eqref{e:LLE}, one finds that the forward-modulated perturbation $\vt$ satisfies the equation
\begin{equation} \label{e:eqformod}
\left(\partial_t-\mathcal{A}\left[\mathring{\phi}\right]\right)\vt = \widetilde{\mathcal{N}}\left[\mathring{\phi}\right]\left(\vt\right) + \widetilde{\mathcal R}\left(\gamma,\partial_t \gamma,\partial_t \sigma\right),
\end{equation}
for $t \in [0,\tau_{\max})$, where $\mathring{\phi}$ denotes the modulated periodic wave
\[
\mathring{\phi}(x,t) = \phi\left(x + \gamma(x,t) + \frac{1}{N}\sigma(t)\right),
\]
with $\sigma$ and $\gamma$ chosen as in~\eqref{e:intsigma}-\eqref{e:intgamma},
the linear operator $\mathcal{A}[\phi]$ is defined by~\eqref{e:Aphi}, the nonlinearity $\smash{\widetilde{\mathcal N}}[\phi](\vt)$ is given by~\eqref{e:Nphi}, and the residual $\widetilde{\mathcal R}(\gamma,\gamma_t,\sigma_t)$ is defined by
\begin{align*}
\widetilde{\mathcal R}(\gamma,\gamma_t,\sigma_t) &= - \beta J\left(\phi''\left(\cdot + \gamma(\cdot,t) + \frac{1}{N}\sigma(t)\right)\left(2\gamma_x + \gamma_x^2\right) + \phi'\left(\cdot + \gamma(\cdot,t) + \frac{1}{N}\sigma(t)\right)\gamma_{xx}\right)\\
&\qquad + \, \phi'\left(\cdot + \gamma(\cdot,t) + \frac{1}{N}\sigma(t)\right) \left(\gamma_t + \frac{1}{N}\sigma_t\right).
\end{align*}
One observes that the equation~\eqref{e:eqformod} for the forward-modulated perturbation $\vt$ arises from the equation~\eqref{e:lin} for the unmodulated perturbation by replacing the periodic steady wave $\phi$ by the modulated wave $\mathring{\phi}$ and adding the residual term $\widetilde{\mathcal R}(\gamma,\gamma_t,\sigma_t)$, which does not depend on $\vt$. In particular, equation~\eqref{e:eqformod} is semilinear in $\vt$, which simplifies acquiring a nonlinear damping estimate. In fact, for the case of localized perturbations, a nonlinear damping estimate has been obtained for equation~\eqref{e:lin} in~\cite[Appendix~A]{HJPR} and, using an analogous approach, for equation~\eqref{e:eqformod} in~\cite[Section~5.4]{ZUM22}. The method transfers to the case of subharmonic perturbations in an $N$-uniform way, leading to the following result.  

\begin{proposition}[Nonlinear Damping Estimate on the Forward-Modulated Perturbation]   \label{P:damping1}
Let $\psi$ be as in Proposition~\ref{p:local_unmod} and $\sigma,\gamma$ and $\tau_{\max}$ as in Proposition~\ref{p:gamma}. Then, there exist $N$-independent constants $R_1, C > 0$ such that the forward-modulated perturbation $\vt(t)$ given by~\eqref{e:defforpert} obeys the estimate
\begin{align} \label{e:dampingineq}
\begin{split}
\|\vt(t)\|_{H^2_N}^2 &\leq C\re^{-t} \|v_0\|_{H^2_N}^2 + C\|\vt(t)\|_{L^2_N}^2\\ 
&\qquad + \, C\int_0^t \re^{-(t-s)} \left(\|\vt(s)\|_{L^2_N}^2 + \|\gamma_x(s)\|_{H^3_N}^2 + \|\partial_s \gamma(s)\|_{H^2_N}^2 + |\partial_s \sigma(s)|^2\right) \de s,
\end{split}
\end{align}
for any $t \in [0,\tau_{\max})$ with
\begin{align}\label{e:upbound}
\sup_{0 \leq s \leq t} \left(\|\vt(s)\|_{H^2_N} + \|\gamma_x(s)\|_{H^3_N} + \|\partial_s \gamma(s)\|_{H^2_N} + |\partial_s \sigma(s)|\right) \leq R_1.
\end{align}
\end{proposition}
\begin{proof} We proceed as in~\cite[Appendix~A]{HJPR} and define the energy
\begin{equation}\label{e:energy}
    E(t) = \left\|\vt_{xx}(t)\right\|_{L^2_N}^2 - \frac{1}{2\beta} \left\langle \mathcal{J} M[\mathring{\phi}(t)] \vt_x(t), \vt_x(t)\right\rangle_{L^2_N},
  \end{equation}
for $t \in [0,\tau_{\max})$, with $M[\phi]$ given by
\begin{align*}
M[\phi] =  2\left(\begin{array}{cc} -2\phi_r\phi_i & \phi_r^2 -\phi_i^2 \\ \phi_r^2 - \phi_i^2 & 2\phi_r \phi_i\end{array}\right),
\end{align*}
where we recall the notation $\phi = \phi_r + \ri \phi_i$. 

{First, using the basic Sobolev interpolation inequality
\begin{align} \label{e:Sobolev} \|f'\|_{L^2_N}^2 \leq \|f''\|_{L^2_N}\|f\|_{L^2_N}, \end{align}
for $f \in H^2_N$, which follows from integration by parts, together with}
Young's and Cauchy-Schwarz inequalities and boundedness of $\phi$, we obtain an $N$-independent constant $K > 0$ such that
\begin{align} \left\|\vt_{xx}(t)\right\|_{L^2_N}^2 \leq 2E(t) + K\left\|\vt(t)\right\|_{L^2_N}^2, \label{e:NLdamp1}\end{align}
for $t \in [0,\tau_{\max})$. {This} shows that the second derivative of $\vt(t)$ is controlled by the energy $E(t)$ and the $L^2_N$-norm of $\vt(t)$.

Next, we use the density of the subspace $H^4_N$ in $H^2_N$ to derive an inequality for the energy. Thus, we take $v_0 \in H^4_N$, for which standard local existence theory (as in Proposition~\ref{p:local_unmod}) implies that $\psi \in C\big([0,T_{\max}),H^4_N\big) \cap C^1\big([0,T_{\max}),H^2_N\big)$. Combining this with Proposition~\ref{p:gamma} yields $\vt \in  C\big([0,\tau_{\max}),H^4_N\big) \cap C^1\big([0,\tau_{\max}),H^2_N\big)$. We denote 
\begin{align*} 
B[\phi] = \left(\begin{array}{cc} 3\phi_{r}^2 + \phi_{i}^2 & 2\phi_{r}\phi_{i} \\
   2\phi_{r}\phi_{i} & \phi_{r}^2 + 3\phi_{i}^2\end{array}\right),
\end{align*}
and differentiate the energy given by~\eqref{e:energy} to obtain
\begin{align*}
\partial_t E(t) = -2 E(t) + E_1(t) + E_2(t) + E_3(t),
\end{align*}
for $t \in [0,\tau_{\max})$, where
\begin{align*}
E_1(t) &= -\frac{1}{\beta}\left\langle \mathcal{J} M[\mathring{\phi}(t)] \vt_x(t), \vt_x(t)\right\rangle_{L^2_N} - \frac{1}{2\beta} \left\langle \mathcal{J} M'[\mathring{\phi}(t)] \mathring{\phi}_t(t) \vt_x(t), \vt_x(t)\right\rangle_{L^2_N}\\
& \qquad + \, 2\Re\left\langle \mathcal{J} \left(\partial_{xx} \left(B[\mathring{\phi}(t)]\vt(t)\right) - B[\mathring{\phi}(t)]\vt_{xx}(t)\right),\vt_{xx}(t) \right\rangle_{L^2_N}\\
&\qquad - \, \Re \left\langle M'[\mathring{\phi}(t)] \mathring{\phi}_x(t) \vt_x(t), \vt_{xx}(t)\right\rangle_{L^2_N}\\
&\qquad + \, \frac{1}{\beta} \Re\left\langle \mathcal{J}M[\mathring{\phi}(t)]\partial_x \left(\left(I + \mathcal{J}(\alpha - B[\mathring{\phi}(t)])\right)\vt(t)\right),\vt_x(t)\right\rangle_{L^2_N},
\end{align*}
contains all irrelevant bilinear terms in $\vt$,
\begin{align*}
E_2(t) = 2\Re\left\langle \partial_x^2 \widetilde{\mathcal{N}}(\vt(t)), \vt_{xx}(t)\right\rangle_{L^2_N} - \frac{1}{\beta}\Re\left\langle \mathcal{J}M[\mathring{\phi}]\partial_x \widetilde{\mathcal{N}}(\vt(t)),\vt_x(t)\right\rangle_{L^2_N},
\end{align*}
consists of all higher-order nonlinear terms in $\vt$, and
\begin{align*}
E_3(t) &= 2\Re\left\langle \partial_x^2 \widetilde{\mathcal{R}}(\gamma(t),\partial_t \gamma(t), \partial_t \sigma(t)), \vt_{xx}(t)\right\rangle_{L^2_N}\\
&\qquad -\, \frac{1}{\beta} \Re \left\langle \mathcal{J}M[\mathring{\phi}]\partial_x \widetilde{\mathcal{R}}(\gamma(t),\partial_t \gamma(t), \partial_t \sigma(t)),\vt_x(t)\right\rangle_{L^2_N},
\end{align*}
contains all residual linear terms in $\vt$.

We estimate the terms $E_j(t)$ with the aid of estimates~\eqref{e:Sobolev} and~\eqref{e:NLdamp1}, the Cauchy-Schwarz and Young inequalities, boundedness of $\partial_x^l \phi$ for $l\in\NM_0$, and the $N$-uniform embedding $H^1_N(\R) \hookrightarrow L^\infty(\R)$. That is, we establish $N$-independent constants $R_1,C_l > 0$, $l = 1,\ldots,5$ such that for each $t \in [0,\tau_{\max})$ satisfying~\eqref{e:upbound} we have
  \begin{align*} |E_1(t)| &\leq C_1 \|\vt(t)\|_{H^2_N}\|\vt(t)\|_{H^1_N} \leq \frac{1}{6}\left\|\vt_{xx}(t)\right\|_{L^2_N}^2 + C_2\left\|\vt(t)\right\|_{L^2_N}^2
\\
  &  \leq \frac{1}{3}E(t) + \left(C_2 + K\right)\left\|\vt(t)\right\|_{L^2_N}^2,\\
|E_2(t)| &\leq C_3\|\vt(t)\|_{H^2_N}^3 \leq \frac{1}{6}\left\|\vt_{xx}(t)\right\|_{L^2_N}^2 + C_3\left\|\vt(t)\right\|_{L^2_N}^2 
\\
  &\leq \frac{1}{3}E(t) + (C_3 + K)\left\|\vt(t)\right\|_{L^2_N}^2,\\
|E_3(t)| &\leq C_4 \left(\|\gamma_x(t)\|_{H^3_N} + \|\partial_t \gamma(t)\|_{H^2_N} + \frac{1}{N}\|\partial_t \sigma(t)\|_{L^2_N}\right)\|\vt(t)\|_{H^2_N}\\ &\leq \frac{1}{6}\left\|\vt_{xx}(t)\right\|_{L^2_N}^2 + C_5\left(\left\|\vt(t)\right\|_{L^2_N}^2 + \|\gamma_x(t)\|_{H^3_N}^2 + \|\partial_t \gamma(t)\|_{H^2_N}^2 + |\partial_t \sigma(t)|^2\right)\\
&\leq \frac{1}{3}E(t) + (C_5+K)\left(\left\|\vt(t)\right\|_{L^2_N}^2 + \|\gamma_x(t)\|_{H^3_N}^2 + \|\partial_t \gamma(t)\|_{H^2_N}^2 + |\partial_t \sigma(t)|^2\right)\end{align*}
for $t \in [0,\tau_{\max})$. Hence, we obtain an $N$-independent constant $C_0 > 0$ such that for $t \in [0,\tau_{\max})$ satisfying~\eqref{e:upbound} we have the energy estimate
\begin{align*}
\partial_t E(t) \leq -E(t) + C_0\left(\left\|\vt(t)\right\|_{L^2_N}^2 + \|\gamma_x(t)\|_{H^3_N}^2 + \|\partial_t \gamma(t)\|_{H^2_N}^2 + |\partial_t \sigma(t)|^2\right).
\end{align*}
Integrating the latter and using~\eqref{e:NLdamp1}, we arrive at
\begin{align*}
\left\|\vt_{xx}(t)\right\|_{L^2_N}^2 &\leq 2\re^{-t}E(0) + K\left\|\vt(t)\right\|_{L^2_N}^2\\ 
&\qquad + 2C_0\int_0^t \re^{-(t-s)} \left(\left\|\vt(s)\right\|_{L^2_N}^2 + \|\gamma_x(s)\|_{H^3_N}^2 + \|\partial_s \gamma(s)\|_{H^2_N}^2 + |\partial_s \sigma(s)|^2\right)\de s,
\end{align*}
for $t \in [0,\tau_{\max})$ satisfying~\eqref{e:upbound}, which, noting that $2E(0) \leq C_*\|\vt(0)\|_{H^2_N} = C_*\|v_0\|_{H^2_N}$ for some $N$-independent constant $C_* > 0$, yields~\eqref{e:dampingineq} {for} $v_0 \in H^4_N$. 

Finally, {for} $v_0 \in H^2_N$, we approximate $v_0$ in $H^2_N$-norm by a sequence $(v_{0,n})_{n \in \NM}$ in $H^4_N$ and note that continuity with respect to initial data, cf.~\cite[Proposition~4.3.7]{CA98}, implies that the perturbed solution $\psi$ of~\eqref{e:LLE} with initial condition $\psi(0) = \phi + v_0$ is, for any $T < T_{\max}$, {approximated in $C([0,T],H^2_N)$} by the sequence of solutions $(\psi_n)_{n \in \NM}$ of~\eqref{e:LLE} with initial data $\psi_n(0) = \phi + v_{0,n}$. Observing that~\eqref{e:dampingineq} only depends on the $H^2_N$-norm of $\vt(t) = \psi(t) - \mathring{\phi}(t)$, the desired result follows by approximation.
\end{proof}

It was exploited in~\cite{ZUM22} for the case of localized perturbations that a nonlinear damping estimate on the forward-modulated perturbation yields nonlinear damping of the inverse-modulated perturbation by using that the $H^k$-norms of the inverse- and forward-modulated perturbations are equivalent (modulo absorbable errors) for any $k \in \NM_0$. In our nonlinear argument in the upcoming section we adopt a similar approach. To this end, we establish the norm equivalence ($N$-uniformly) in the current subharmonic setting following~\cite[Lemma~5.1]{ZUM22} and~\cite[Lemma~2.7]{JNRZ14}.

\begin{lemma} \label{lem:equivalence}
Let $\psi$ be as in Proposition~\ref{p:local_unmod} and $\sigma,\gamma$ and $\tau_{\max}$ as in Proposition~\ref{p:gamma}. Then, there exists an $N$-independent constant $C > 0$ such that the inverse- and forward-modulated {perturbations $v$ and $\vt$} defined by~\eqref{e:definvmod} and~\eqref{e:defforpert}, respectively, satisfy
\begin{align*}
\|v(t)\|_{H^2_N} \leq C\left(\|\vt(t)\|_{H^2_N} + \|\gamma_x(t)\|_{H^1_N}\right), \quad \|\vt(t)\|_{L^2_N} \leq C\left(\|v(t)\|_{L^2_N} + \|\gamma_x(t)\|_{H^1_N}\right),
\end{align*}
for any $t \in [0,\tau_{\max})$.
\end{lemma}
\begin{proof} 
Take $t \in [0,\tau_{\max})$. By Proposition~\ref{p:gamma} there exists an $N$-independent constant $r_0 > 0$ such that
\begin{equation}\label{e:R}
  \|\gamma(t)\|_{H^3_N} + |\sigma(t)| < r_0, \qquad \|\gamma_x(t)\|_{L^\infty} \leq \frac12.
\end{equation}
As in the proof of Proposition~\ref{C:local_v} we consider the map $A_t \colon \R \to \R$ given by
  \[
  A_t(x) = x - \gamma(x,t) - \frac{1}{N} \sigma(t),
  \]
which is invertible by~\eqref{e:R} and satisfies
\begin{equation}\label{e:AtNT}
A_t(NT) - A_t(0) = NT,
\end{equation}
by the $NT$-periodicity of $\gamma(\cdot,t)$. Using the equalities
    \[
    x = A_t(A_t^{-1}(x)) = A_t^{-1}(x) - \gamma(A_t^{-1}(x),t) - \frac{1}{N}\sigma(t)
    \]
and the mean value theorem, we find that
\begin{align} \label{e:Atest}
\begin{split}
\left|A_t^{-1}(x) - \left(x + \gamma(x,t) + \frac{1}{N}\sigma(t)\right)\right| &= \left|\gamma(A_t^{-1}(x),t) - \gamma(x,t)\right| \\ 
&\leq \|\gamma_x(t)\|_{L^\infty} \left|A_t^{-1}(x) - x\right|\\ 
&\leq \|\gamma_x(t)\|_{L^\infty} \left(\left|\gamma\left(A_t^{-1}(x),t\right)\right| +  \frac{1}{N} |\sigma(t)|\right),
\end{split}
\end{align}
for all $x \in \R$. Moreover, the inverse function theorem implies
\begin{align} 
\label{e:derid}
\partial_x \left(A_t^{-1}(x)\right) = \frac{1}{1-\gamma_x\left(A_t^{-1}(x),t\right)}, \quad \partial_x^2 \left(A_t^{-1}(x)\right) = \frac{\gamma_{xx}\left(A_t^{-1}(x),t\right)}{\left(1-\gamma_x\left(A_t^{-1}(x),t\right)\right)^3},\end{align}
for all $x \in \R$. Combing the latter with {\eqref{e:R} and~\eqref{e:Atest} we find\footnote{Throughout this proof the notation $A\lesssim B$ means that there exists an $N$- and $t$-independent constant $K>0$ such that $A\leq K B$.}}
\begin{align} \label{e:Atest2}
\left|\partial_x \left(A_t^{-1}(x)\right) - 1\right| \lesssim \left|\gamma_x\left(A_t^{-1}(x),t\right)\right|, \quad \left|\partial_{xx} \left(A_t^{-1}(x)\right)\right| \lesssim\left|\gamma_{xx}\left(A_t^{-1}(x),t\right)\right|,
\end{align}
for all $x \in \R$.

Using the properties~\eqref{e:AtNT} and~\eqref{e:derid} of the map $A_t$, the $N$-uniform embedding $H^1_N \hookrightarrow L^\infty(\R)$, Young's inequality, and the bounds~\eqref{e:R} we find the inequalities
\begin{align} 
\label{e:subst22}
\begin{split}
\left\|f\left(A_t^{-1}(\cdot)\right)\right\|_{L^2_N}^2
&=  \int_{A_t(0)}^{A_t(NT)} \left|f\left(A_t^{-1}(y)\right)\right|^2 \de y  
= \int_{0}^{NT} \left|f\left(x\right)\right|^2 A_t'(x) \de x 
\lesssim\|f\|_{L^2_N}^2,\\
\left\|f\left(A_t(\cdot)\right)\right\|_{L^2_N}^2
&=  \int_{0}^{NT} \left|f\left(A_t(y)\right)\right|^2 \de y  
= \int_{A_t(0)}^{A_t(NT)} \frac{\left|f\left(x\right)\right|^2}{1-\gamma_x(A_t^{-1}(x),t)} \de x \lesssim\|f\|_{L^2_N}^2,
\end{split}
\end{align}
for $f \in L^2_N$, and
\begin{align} 
\label{e:subst}
\begin{split}
\left\|f\left(A_t^{-1}(\cdot)\right)\right\|_{H^2_N}^2 &\lesssim \int_{A_t(0)}^{A_t(NT)} \left|f\left(A_t^{-1}(y)\right)\right|^2 \de y
+ \int_{A_t(0)}^{A_t(NT)} \left|f''\left({A_t^{-1}}(y)\right)\right|^2 \left|\partial_y\left(A_t^{-1}(y)\right)\right|^4 \de y\\ 
&\qquad\quad + \, \int_{A_t(0)}^{A_t(NT)} \left|f'\left({A_t^{-1}}(y)\right)\right|^2 \left(\left|\partial_y\left(A_t^{-1}(y)\right)\right|^2 + \left|\partial_{y}^2\left(A_t^{-1}(y)\right)\right|^2\right) \de y\\
&= \int_{0}^{NT} \left|f\left(x\right)\right|^2 A_t'(x) \de x
+ \int_{0}^{NT} \frac{\left|f''\left(x\right)\right|^2}{\left(1-\gamma_x(x,t)\right)^3} \de x\\ 
&\qquad\quad + \,\int_{0}^{NT} \frac{\left|f'\left(x\right)\right|^2}{1-\gamma_x(x,t)}\left(1 + \frac{\gamma_{xx}(x,t)^2}{\left(1-\gamma_x(x,t)\right)^4}\right) \de x \\ 
&\lesssim\|f\|_{H^2_N}^2.
\end{split} 
\end{align}
for $f \in H^2_N$. 

{Recalling the formulas~\eqref{e:definvmod} and~\eqref{e:defforpert} for the inverse- and forward-modulated perturbations and applying}
the mean value theorem, Young's inequality, the $N$-uniform embedding $H^1_N \hookrightarrow L^\infty(\R)$, the bounds~\eqref{e:R}, and the estimates~\eqref{e:Atest},~\eqref{e:Atest2} and~\eqref{e:subst22}, we find 
\begin{align*} 
\left\|\vt(\cdot,t) - v\left(A_t^{-1}(\cdot),t\right)\right\|_{L^2_N} 
&\lesssim\|\phi'\|_{L^\infty} \|\gamma_x(t)\|_{L^\infty} \left(\left\|\gamma\left(A_t^{-1}(\cdot),t\right)\right\|_{L^2_N} + \frac{1}{N} \|\sigma(t)\|_{L^2_N}\right) \lesssim \|\gamma_x(t)\|_{H^1_N},\\
\left\|\vt(\cdot,t) - v\left(A_t^{-1}(\cdot),t\right)\right\|_{H^2_N} 
&\lesssim\|\phi'\|_{W^{2,\infty}} \left(\|\gamma_x(t)\|_{L^\infty} \left(\left\|\gamma\left(A_t^{-1}(\cdot),t\right)\right\|_{L^2_N} + \frac{1}{N} \|\sigma(t)\|_{L^2_N}\right)\right.\\ 
&\quad\left. \phantom{\frac{1}{N}} + \left\|\gamma_x(t)\right\|_{H^1_N} + \left\|\gamma_x\left(A_t^{-1}(\cdot),t\right)\right\|_{L^2_N} + \left\|\gamma_{xx}\left(A_t^{-1}(\cdot),t\right)\right\|_{L^2_N}\right)\\
&\lesssim\|\gamma_x(t)\|_{H^1_N}, 
\end{align*}
Subsequently applying~\eqref{e:subst2} yields
\begin{align*}
\left\|\vt(A_t(\cdot),t) - v\left(\cdot,t\right)\right\|_{H^2_N}
\lesssim \left\|\vt(\cdot,t) - v\left(A_t^{-1}(\cdot),t\right)\right\|_{H^2_N} \lesssim\|\gamma_x(t)\|_{H_N^1}.\end{align*}
Finally, combining the latter two estimates with~\eqref{e:subst2},~\eqref{e:subst22} and~\eqref{e:subst}, we obtain
\begin{align*}
  \|v(t)\|_{L^2_N} &
  \lesssim \left\|\vt(A_t(\cdot),t)\right\|_{L^2_N} + \left\|\gamma_x(t)\right\|_{H^1_N}
  \lesssim \|\vt(t)\|_{L^2_N} + \|\gamma_x(t)\|_{H^1_N},\\
  \|\vt(t)\|_{H^2_N} &
  \lesssim \left\|v\left(A_t^{-1}(\cdot),t\right)\right\|_{H^2_N} + \|\gamma_x(t)\|_{H^1_N}
  \lesssim \|v(t)\|_{H^2_N} + \|\gamma_x(t)\|_{H^1_N},
\end{align*}
which concludes the proof.
\end{proof}

\section{Nonlinear Stability Analysis} \label{S:proof}

We prove our main result, Theorem~\ref{T:main}, by {taking
  \begin{equation}\label{e:gammanl}
    \gamma_{\mathrm{nl}}(x,t) = \gamma(x,t) +\frac1N \sigma(t),
    \end{equation}
and a suitably chosen constant $\sigma_{\mathrm{nl}}$.
We apply} the linear estimates, stated in Proposition~\ref{P:lin}, to the nonlinear iteration scheme, which was established in~\S\ref{S:nonlinear} and consists of equations for the inverse-modulated perturbation $v$ and the phase modulation functions $\sigma$ and $\gamma$. We use the nonlinear damping estimate on the forward-modulated perturbation $\vt$, obtained in Proposition~\ref{P:damping1}, as well as the connection between the norms of the inverse- and forward-modulated perturbations established in Lemma~\ref{lem:equivalence}, to control regularity in the iteration scheme.

\begin{proof}[Proof of Theorem~\ref{T:main}]
We start by defining a template function, which controls the phase modulation functions $\gamma \colon [0,\tau_{\max}) \to H^4_N$, $\sigma \colon [0,\tau_{\max}) \to \R$ and the forward-modulated perturbation $\vt \colon [0,\tau_{\max}) \to H^2_N$ in the nonlinear argument. By Proposition~\ref{p:gamma} and Corollary~\ref{C:local_vf}, the template function $\eta \colon [0,\tau_{\max}) \to \R$ given by
\begin{align*}
\eta(t) &= \sup_{0\leq s\leq t} \left[(1+s)^{\frac{3}{4}}\left(\|\vt(s)\|_{H^2_N} + \left\|\partial_x \gamma(s)\right\|_{H^3_N} + \left\|\partial_s \gamma(s)\right\|_{H^2_N}\right)\right.\\ 
&\qquad\qquad\qquad\qquad \left. +\, (1+s)^{\frac{1}{4}}\left\|\gamma(s)\right\|_{L^2_N} + (1+s)^{\frac32} \left|\partial_s \sigma(s)\right| + |\sigma(s)|\right],
\end{align*}
is continuous, positive and monotonically increasing.

As usual, {the key step of the} approach is to prove that there exist $N$- and $t$-independent constants $R > 0$ and $C \geq 1$ such that for all $t \in [0,\tau_{\max})$ with $\eta(t) \leq R$ we have the inequality
  \begin{align}
    \eta(t) \leq C\left(E_0 + \eta(t)^2\right). \label{e:etaest}
  \end{align}
To this end, we take $R := \min\{\tfrac12 r_0,R_1\}$ with $r_0 > 0$ as in Proposition~\ref{p:gamma} and $R_1 > 0$ as in Proposition~\ref{P:damping1}, and assume $t \in [0,\tau_{\max})$ is such that
\begin{equation}\label{e:eta}
\eta(t) \leq R.
\end{equation}
We note that by Proposition~\ref{p:gamma} it holds
\begin{align} \label{e:gammax}
\|\gamma_x(s)\|_{L^\infty} \leq \frac12,
\end{align}
for all $s \in [0,t]$.

First, we point out that by Proposition~\ref{C:local_v} the inverse-modulated perturbation $v(r)$, given by~\eqref{e:definvmod}, lies in $H^2_N$ for all $r \in [0,t]$. In particular, Lemma~\ref{lem:equivalence} and the inequality~\eqref{e:eta} yield the bound\footnote{Throughout this proof the notation $A\lesssim B$ means that there exists an $N$- and $t$-independent constant $K>0$ such that $A\leq K B$.} 
\begin{align} \label{e:bdH2m}
\|v(r)\|_{H^2_N} \lesssim \frac{\eta(r)}{(1+r)^{\frac34}},
\end{align}
for $r \in [0,t]$. Therefore,~\eqref{e:eta},~\eqref{e:gammax},~\eqref{e:bdH2m} and Lemma~\ref{lem:mod_nonlL2} afford the nonlinear estimate
\begin{align} \label{e:nlest}
\left\|\mathcal{N}(v,\gamma,\partial_r \gamma,\partial_r \sigma)(r)\right\|_{L_N^2 \cap L_N^1} &\lesssim \frac{\eta(r)^2}{(1+r)^{\frac32}}, 
\end{align}
for $r \in [0,t]$. Applying the linear estimates in Proposition~\ref{P:lin} and the nonlinear bound~\eqref{e:nlest} to the Duhamel formulations~\eqref{e:intv2} and~\eqref{e:intgamma2}, we arrive at
\begin{align}
\label{e:duh1}
\|v(s)\|_{L^2_N} &\lesssim \frac{E_0}{(1+s)^{\frac34}} + \int_0^s \frac{\eta(r)^2}{\re^{\mu(s-r)}(1+r)^{\frac32}} \de r + \int_0^s \frac{\eta(r)^2}{(1+s-r)^{\frac34}(1+r)^{\frac32}} \de r \lesssim \frac{E_0 + \eta(s)^2}{(1+s)^{\frac34}},
\end{align}
and
\begin{align}
\label{e:duh2}
\begin{split}
\|\partial_x^\ell \partial_s^j \gamma(s)\|_{L^2_N} &\lesssim \frac{E_0}{(1+s)^{\frac34}} + \int_0^s \frac{\eta(r)^2}{(1+s-r)^{\frac34}(1+r)^{\frac32}} \de r \lesssim \frac{E_0 + \eta(s)^2}{(1+s)^{\frac34}},\\
\|\gamma(s)\|_{L^2_N} &\lesssim \frac{E_0}{(1+s)^{\frac14}} + \int_0^s \frac{\eta(r)^2}{(1+s-r)^{\frac14}(1+r)^{\frac32}} \de r \lesssim \frac{E_0 + \eta(s)^2}{(1+s)^{\frac14}},
\end{split}
\end{align}
for all $s \in [0,t]$ and $\ell, j \in \NM_0$ with $1 \leq \ell + 2j \leq 4$. On the other hand, we use the Cauchy-Schwarz inequality and the properties $\|\chi\|_{L^\infty} = 1$ and $\chi'(s) = 0$ for $s \in \R \setminus [1,2]$ and the estimates~\eqref{e:eta} and~\eqref{e:nlest} to bound the right-hand side of~\eqref{e:intsigma2} as
\begin{align}
\label{e:duh3}
\begin{split}
|\sigma(s)| &\lesssim E_0 + \int_0^s \frac{\eta(r)^2}{(1+r)^{\frac32}} \de r \lesssim E_0 + \eta(s)^2,\\
|\partial_t \sigma(s)| &\lesssim |\chi'(s)| E_0 + \int_0^s \frac{|\chi'(s-r)|\eta(r)^2}{(1+r)^{\frac32}} \de r \lesssim \frac{E_0 + \eta(s)^2}{(1+s)^{\frac32}},
\end{split}
\end{align}
for all $s \in [0,t]$. Next, we combine Lemma~\ref{lem:equivalence} with~\eqref{e:duh1} and~\eqref{e:duh2} to arrive at
\begin{align}
\label{e:duh4}
\|\vt(s)\|_{L^2_N} \lesssim \frac{E_0 + \eta(s)^2}{(1+s)^{\frac34}},
\end{align}
for $s \in [0,t]$. Finally, by Proposition~\ref{P:damping1} and the inequalities~\eqref{e:duh2},~\eqref{e:duh3} and~\eqref{e:duh4} we obtain
\begin{align} \label{e:damping}
\|\vt(t)\|_{H^2_N}^2 &\lesssim \re^{-t} E_0^2 + \frac{\left(E_0 + \eta(t)^2\right)^2}{(1+t)^{\frac{3}{2}}} + \int_0^t \frac{\re^{-(t-s)} \left(E_0 + \eta(s)^2\right)^2}{(1+s)^{\frac{3}{2}}} \de s \lesssim \frac{\left(E_0 + \eta(t)^2\right)^2}{(1+t)^{\frac{3}{2}}}.
\end{align}
Hence, combining the inequalities~\eqref{e:duh2},~\eqref{e:duh3} and~\eqref{e:damping} yields an $N$- and $t$-independent constant $C \geq 1$ such that the key inequality~\eqref{e:etaest} is satisfied.

{To end the proof of Theorem~\ref{T:main},} we set $\eps = \min\{\frac{1}{4C^2},\frac{R}{2C}\} > 0$ and take $E_0 \in (0,\eps)$. Then, as outlined in the proof of~\cite[Theorem~1.3]{HJPR}, inequality~\eqref{e:etaest} yields $\eta(t) \leq 2CE_0 \leq R$ for all $t \in [0,\tau_{\max})$. Consequently,~\eqref{e:gammablowup} cannot hold and we have $\tau_{\max} = T_{\max}$ by Proposition~\ref{p:gamma}. Furthermore, the mean value theorem implies 
\begin{align} \label{e:psib}
\begin{split}
\|\psi(t)\|_{H^2_N} &\lesssim \|\vt(t)\|_{H^2_N} + \left\|\phi\left(\cdot + \gamma(\cdot,t) + \frac{1}{N}\sigma(t)\right) - \phi(\cdot)\right\|_{H^2_N} + \left\|\phi\right\|_{H^2_N}\\
&\lesssim \|\vt(t)\|_{H^2_N} + \left\|\phi'\right\|_{W^{2,\infty}} \left(1 + \|\gamma(t)\|_{H^2_N} + \frac{1}{\sqrt{N}}|\sigma(t)|\right) + \|\phi\|_{H^2_N},
\end{split}
\end{align}
for $t \in [0,\tau_{\max})$. Hence, by~\eqref{e:psib} and the fact that $\eta(t) \leq R$ for all $t \in [0,\tau_{\max})$,~\eqref{e:psiblowup} cannot hold and Proposition~\ref{p:local_unmod} yields $\tau_{\max} = T_{\max} = \infty$. We conclude that we have 
\begin{align} \label{e:etaest2}
\eta(t) \leq 2CE_0 \leq R,
\end{align}
for all $t \geq 0$, which yields the last two estimates in~\eqref{e:MTmodder} with $\gamma_{\mathrm{nl}}$ as defined in~\eqref{e:gammanl}. In addition, the mean value theorem affords the inequality 
\begin{align*} 
\left\|\psi(t) - \phi\right\|_{H^2_N} &\leq \|\vt(t)\|_{H^2_N} + \|\phi'\|_{W^{2,\infty}}\left(\|\gamma(t)\|_{H^2_N} + \frac{1}{\sqrt{N}} |\sigma(t)|\right),
\end{align*}
for $t \geq 0$, where we use~\eqref{e:etaest2}. Combining the latter with~\eqref{e:etaest2} proves the first estimate in~\eqref{e:MTmodder}. 

Finally, {we set
\[
 \sigma_{\mathrm{nl}} = \LA\widetilde{\Phi}_0,v_0\RA_{L^2_N} + \int_0^\infty \LA\widetilde{\Phi}_0,\mathcal{N}(v,\gamma,\partial_s \gamma,\partial_s \sigma)(s)\RA_{L^2_N} \de s,
\]
which is well-defined and satisfies $|\sigma_{\mathrm{nl}}| \lesssim E_0$ by the Cauchy-Schwarz inequality and the estimates~\eqref{e:nlest} and~\eqref{e:etaest2}. Using the Cauchy-Schwarz inequality, the properties} $\|\chi\|_{L^\infty} = 1$ and $\chi(t) = 1$ for $t \in [2,\infty)$, and estimates~\eqref{e:nlest} and~\eqref{e:etaest2} we obtain
\begin{align*}
|\sigma(t) - \sigma_{\mathrm{nl}}| \lesssim \int_{t-2}^\infty \frac{\eta(s)^2}{(1+s)^{\frac32}} \de s \lesssim \frac{E_0}{(1+t)^{\frac12}}, 
\end{align*}
for $t \geq 2$. On the other hand, the mean value theorem and~\eqref{e:etaest2} yield
\begin{align*}
\left\|\psi(t) - \phi\left(\cdot + \frac{1}{N}\sigma_{\mathrm{nl}} \right)\right\|_{H^2_N} & \lesssim \left\|\psi(t) - \phi\left(\cdot + \frac{1}{N}\sigma(t) \right)\right\|_{H^2_N} + \frac{1}{\sqrt{N}}\|\phi'\|_{W^{2,\infty}} |\sigma(t) - \sigma_{\mathrm{nl}}|\\
&\lesssim \|\vt(t)\|_{H^2_N} + \|\phi'\|_{W^{2,\infty}} \|\gamma(t)\|_{H^2_N}
+ \frac{1}{\sqrt{N}}\|\phi'\|_{W^{2,\infty}} |\sigma(t) - \sigma_{\mathrm{nl}}|,
\end{align*}
for $t \geq 0$. The last two estimates and~\eqref{e:etaest2} justify the remaining inequalities in~\eqref{e:MTmodder}, and complete the proof.
\end{proof}

Finally, it remains to prove Corollary~\ref{C:combined}.  
\begin{proof}[Proof of Corollary~\ref{C:combined}]
Let $\phi$, $\varepsilon$ and $M$ be as in Theorem~\ref{T:main}.  Fix $N\in\NM$, let $\delta_N$ be as in~\eqref{e:deltaN}, and, for $\delta\in(0,\delta_N)$, let $\eps_\delta$ be as in Theorem~\ref{T:SS}. Let $v_0\in H^2_N$ satisfy $E_0 := \|v_0\|_{H^2_N \cap L^1_N} <\max\{\eps,\eps_\delta\}$.

If $\eps_\delta > \eps$, we take $T_\delta=0$, and the proof is finished by Theorem~\ref{T:SS}. Otherwise, Theorem~\ref{T:main} yields a constant $\sigma_{\mathrm{nl}} \in \R$ and a global mild solution $\psi \in C\big([0,\infty),H^1_N\big)$ of~\eqref{e:LLE} satisfying 
\begin{align} \label{e:eees}
\left\|\psi(\cdot,t) - \phi\left(\cdot + \frac{1}{N}\sigma_{\mathrm{nl}} \right)\right\|_{H^1_N} \leq ME_0 (1+t)^{-\frac{1}{4}},
\end{align}
for all $t>0$. Thus, there exists a (minimal) time $T_\delta\geq0$ such that $ME_0(1+t)^{-1/4} < \eps_\delta$ for $t \geq T_\delta$ so that 
\begin{align} \label{e:ees}
\left\|\psi\left(\cdot - \frac{1}{N} \sigma_{\rm nl},t\right) - \phi\left(\cdot\right)\right\|_{H^1_N} = \left\|\psi(\cdot,t) - \phi\left(\cdot + \frac{1}{N}\sigma_{\mathrm{nl}} \right)\right\|_{H^1_N}< \eps_\delta,
\end{align}
for all $t\geq T_\delta$. If we take the initial datum $\widetilde{\psi}_0=\psi(\cdot - \frac{1}{N}\sigma_{\rm nl},T_\delta)\in H^2_N,$ then the perturbation $\smash{\widetilde{v}_0 := \widetilde{\psi}_0 - \phi \in H^1_N}$ satisfies $\smash{\|\widetilde{v}_0\|_{H^1_N}<\eps_\delta}$ by~\eqref{e:ees}. By uniqueness of solutions, cf.~Proposition~\ref{p:local_unmod}, the solution $\smash{\widetilde{\psi}}$ of~\eqref{e:LLE} with initial condition $\smash{\widetilde{\psi}(0) = \widetilde{\psi}_0}$ satisfies $\smash{\widetilde{\psi}(x,t) = \psi(x - \frac{1}{N}\sigma_{\rm nl},t + T_\delta)}$ for $x \in \R$ and $t \geq 0$. On the other hand, Theorem~\ref{T:SS} yields constants $\tilde\gamma \in \R$ and $C_\delta > 0$ such that
\begin{align} \label{e:es}
\begin{split}
\left\|\psi(\cdot,t + T_\delta) - \phi\left(\cdot + \frac{\sigma_{\rm nl}}{N} + \tilde\gamma \right)\right\|_{H^1_N} &= \left\|\psi\left(\cdot -  \frac{\sigma_{\rm nl}}{N},t + T_\delta\right) - \phi\left(\cdot + \tilde\gamma \right)\right\|_{H^1_N} \\&= \left\|\widetilde{\psi}(\cdot,t) - \phi\left(\cdot + \tilde\gamma \right)\right\|_{H^1_N}\\
&\leq C_\delta\left\|\widetilde{v}_0\right\|_{H^1_N} \re^{-\delta t} \leq C_\delta M E_0 \re^{\delta T_\delta} \re^{-\delta (t+T_\delta)},
\end{split}
\end{align}
for all $t\geq 0$, where we use that $\|\tilde{v}_0\|_{H^1_N} \leq ME_0(1+T_\delta)^{-\frac14} \leq ME_0$ by~\eqref{e:eees}. Comparing estimates~\eqref{e:eees} and~\eqref{e:es} and letting $t \to \infty$ implies $\smash{\tilde\gamma} = 0$. Finally, taking $M_\delta = \smash{C_\delta M \re^{\delta T_\delta}}$ completes the proof.
\end{proof}
\appendix
\renewcommand*{\thesection}{\Alph{section}}

\section{Proof of Proposition~\ref{P:lin}} \label{app:linear}

{
Assume that $\phi$ is a smooth, $T$-periodic stationary solution of the LLE~\eqref{e:LLE}, which is diffusively spectrally stable in the sense of Definition~\ref{Def:spec_stab}. We briefly recall below the main steps of the linear analysis from
\cite{HJP21} leading to the result in Proposition~\ref{P:lin}.

The starting point of this analysis is the Floquet-Bloch theory for $NT$-periodic functions developed in~\cite{HJP21}. Setting\footnote{The set $\Omega_N$ is the analogue for $NT$-periodic functions of the interval $[-\pi/T,\pi/T)$ in the better known  Floquet-Bloch theory for functions in $L^2(\R)$.}
\[
\Omega_N=\left\{\xi\in[-\pi/T,\pi/T):\re^{\ri\xi NT}=1\right\},
\]
a function $g\in L^2_N$ can be represented by the inverse Bloch formula
\[
g(x)=\frac{1}{NT}\sum_{\xi\in\Omega_N} \re^{\ri\xi x}\mathcal{B}_T(g)(\xi,x),
\]
in which $\mathcal{B}_T$ is the $T$-periodic Bloch transform
defined as
\[
\mathcal{B}_T(g)(\xi,x) =\sum_{\ell\in\ZM}\re^{2\pi \ri\ell x/T}\hat{g}\left(\xi+2\pi\ell/T\right),\qquad\xi\in\Omega_N,~x\in\RM,
\]
where $\hat{g}$ denotes the Fourier transform of $g$ on the torus given by
\[
\hat{g}(z):=\int_{-NT/2}^{NT/2} \re^{-\ri zy}g(y)\de y.
\]
Accordingly, for the operator $\mathcal{A}[\phi]$, acting on $L^2_N$, we have the identity
\[
\mathcal{B}_T\left(\mathcal{A}[\phi]v\right)(\xi,x) = \mathcal{A}_\xi[\phi]\mathcal{B}_T(v)(\xi,x), \quad v\in L^2_N,
\]
in which $\mathcal{A}_\xi[\phi]$, acting on $L^2_{\rm per}(0,T)$, are the associated Bloch operators introduced in Definition~\ref{Def:spec_stab}.

An important consequence of this Floquet-Bloch decomposition is the spectral decomposition
\[
\sigma_{L^2_N}\left(\mathcal{A}[\phi]\right)=\bigcup_{\xi\in\Omega_N}\sigma_{L^2_{\rm per}(0,T)}\left(\mathcal{A}_\xi[\phi]\right),
\]
which characterizes the $L_N^2$-spectrum of $\mathcal{A}[\phi]$ in terms of the union of the eigenvalues (including multiplicities) of the $1$-parameter family of Bloch operators $\left\{\mathcal{A}_\xi[\phi]\right\}_{\xi\in\Omega_N}$.
Furthermore, a direct consequence of the diffusive spectral stability of $\phi$ is that there is an analytic curve $\lambda_c(\xi)$ of simple eigenvalues of the Bloch operators $\mathcal A_\xi[\phi]$, which expands as
\begin{equation}\label{e:lambdac}
\lambda_c(\xi)=\ri a\xi- d \xi^2+\mathcal{O}(|\xi|^3),
\end{equation}
for some $a\in\RM$ and $d>0$, while the rest of the spectrum is bounded away from the imaginary axis. The eigenfunction $\Phi_\xi$ associated with $\lambda_c(\xi)$ is a smooth function, depending analytically on $\xi$, which expands as
\[
\Phi_\xi=\phi'+\mathcal{O}(|\xi|).
\]

As shown in~\cite{HJP21}, both the operator $\mathcal{A}[\phi]$, acting on $L^2_N$, and the associated Bloch operators $\mathcal{A}_\xi[\phi]$, acting on $L^2_{\rm per}(0,T)$, generate $C^0$-semigroups. Furthermore, we have the identity
\[
\mathcal{B}_T\left(\re^{\mathcal{A}[\phi]t}v\right)(\xi,x) =
\left(\re^{\mathcal{A}_\xi[\phi]t}\mathcal{B}_T(v)(\xi,\cdot)\right)(x), \quad v\in L^2_N,
\]
and the representation formula
\begin{equation}\label{e:bloch_soln}
\re^{\mathcal{A}[\phi]t}v(x) = \frac{1}{NT}\sum_{\xi\in\Omega_N}\re^{\ri\xi x}\re^{\mathcal{A}_\xi[\phi]t}\mathcal{B}_T(v)(\xi,x), \quad v\in L^2_N,
\end{equation}
which is used to obtain the decomposition~\eqref{e:semidecomp0} of the semigroup $\re^{\mathcal{A}[\phi]t}$. Without going into details, we only recall the formula of the principal part $s_{p,N}(t)$ of the semigroup $\re^{\mathcal A[\phi] t}$, given by
\begin{equation}\label{e:spN}
s_{p,N}(t)v(x)=\frac{1}{NT}\sum_{\xi\in\Omega_N \setminus \{0\}}\rho(\xi)\re^{\ri\xi x}\re^{\lambda_c(\xi)t}\left\langle\widetilde{\Phi}_\xi,\mathcal{B}_T(v)(\xi,\cdot)\right\rangle_{L_{\rm per}^2(0,T)}, \quad v\in L^2_N,
\end{equation}
in which $\rho$ is a smooth cutoff function satisfying $\rho(\xi)=1$ for $|\xi|<\xi_1/2$ and $\rho(\xi)=0$ for $|\xi|>\xi_1$, with $\xi_1 \in [-\pi/T,\pi/T)$ suitably chosen, and $\widetilde{\Phi}_\xi$ is the smooth eigenfunction of the adjoint operator $\mathcal{A}_\xi[\phi]^*$ associated with the eigenvalue $\smash{\overline{\lambda_c(\xi)}}$, normalized to satisfy $\smash{\left\langle\widetilde{\Phi}_\xi,\Phi_\xi\right\rangle_{L_{\rm per}^2(0,T)}}=1$. We refer to~\cite{HJP21} for the formula for $S_N(t)$ and the proof of its decay property~\eqref{L:mod_bd2}.

The decay properties of $s_{p,N}(t)$ in~\eqref{L:mod_bd1} are obtained by directly estimating the sum on the right hand side of~\eqref{e:spN}, using the expansion~\eqref{e:lambdac} of the eigenvalue $\lambda_c(\xi)$, noting that
\[
\left|\LA\widetilde{\Phi}_\xi, \mathcal{B}_T(v)(\xi,\cdot)\RA_{L_{\rm per}^2(0,T)}\right| 
\lesssim\|v\|_{L_N^1},\quad \left|\rho(\xi)\re^{\ri\xi x}e^{\lambda_c(\xi)t}\right|\lesssim \re^{-d\xi^2 t},
\]
and 
\[
\frac1N\sum_{\xi\in\Omega_N}\xi^{2n}\re^{-2d\xi^2 t} \lesssim (1+t)^{-\frac12-n}.
\]
for $n\in\NM_0$; see~\cite{HJP21} for more details.\footnote{See also~\cite{HJPR} where similar estimates have been obtained in the case of localized perturbations for functions $v\in L^1(\RM)\cap L^2(\RM)$, or see~\cite{JP21} for details on how to get uniform bounds on the Riemann sum.}
}

\section{Local Theory} \label{app:local}

The goal of this appendix is to establish Proposition \ref{p:gamma}, which provides local existence for the phase modulation functions.  
To this end, we first prove the following preliminary result.

\begin{lemma} \label{L:loc_gamma}
Let $\psi$ and $T_{\max}$ be as in Proposition~\ref{p:local_unmod}. The mapping $V \colon L^2_N \times \R \times [0,T_{\max}) \to L^2_N$ given by
\begin{align*}V(\gamma,\sigma,t)[x] &= \psi\left(x-\gamma(x)-\frac{1}{N}\sigma,t\right) - \phi(x),\end{align*}
is well-defined, continuous in $t$, and locally Lipschitz continuous in $(\gamma,\sigma)$ (uniformly in $N$ and $t$ on compact subintervals of $[0,T_{\max})$).
\end{lemma}
\begin{proof}
We apply the mean value theorem to establish 
\begin{align}\label{e:estV1}
\begin{split}
&\left\|V(\gamma_1,\sigma_1,t) - V(\gamma_2,\sigma_2,t)\right\|_{L^2_N} \leq \|\psi(t)\|_{W^{1,\infty}}\left(\|\gamma_1 - \gamma_2\|_{L^2_N} + \frac{1}{\sqrt{N}} |\sigma_1 - \sigma_2|\right)
\end{split}
\end{align}
for $\gamma_{1,2} \in L_N^2$, $\sigma_{1,2} \in \R$ and $t \in [0,T_{\max})$. Hence, recalling the $N$-uniform embedding $H^1_N \hookrightarrow L^\infty(\R)$ and using Proposition~\ref{p:local_unmod}, it follows that $V$ is locally Lipschitz continuous in $(\gamma,\sigma)$ (uniformly in $N$ and $t$ on compact subintervals of $[0,T_{\max})$). In addition, taking $\gamma_2,\sigma_2 = 0$ in~\eqref{e:estV1} and recalling $\phi \in L^2_N$, proves that $V$ is well-defined. 

Similarly as in~\eqref{e:estV1}, we employ the mean value theorem to obtain
\begin{align*}
\|V(\gamma,\sigma,t) - V(\gamma,\sigma,s)\|_{L^{2}_N} &\leq \left\|\left(V(\gamma,\sigma,t) - V(\gamma,\sigma,s)\right) - \left(V(0,0,t) - V(0,0,s)\right)\right\|_{L^{2}_N}\\
&\qquad + \, \|V(0,0,t) - V(0,0,s)\|_{L^2_N}\\
 &\leq \|\psi(t) - \psi(s)\|_{W^{1,\infty}}\left(\|\gamma\|_{L^2_N} + \frac{1}{\sqrt{N}}|\sigma|\right) + \|\psi(t) - \psi(s)\|_{L^2_N},
\end{align*}
for $\gamma \in L^2_N$, $\sigma \in \R$ and $s,t \in [0,T_{\max})$. Continuity of $V$ with respect to $t$ now follows from Proposition~\ref{p:local_unmod} and the embedding $H^1_N \hookrightarrow L^\infty(\R)$.
\end{proof}

Next, we establish the relevant local existence result.

\begin{proposition} % \label{p:loc_gamma}
Let $\psi$ and $T_{\max}$ be as in Proposition~\ref{p:local_unmod}. Moreover, let $V \colon L^2_N \times \R \times [0,T_{\max}) \to L^2_N$ be the mapping defined in Lemma~\ref{L:loc_gamma}. Then, there exist a maximal time $\tau_{\max} \in (0,T_{\max}]$ and an $N$-independent constant $r_0 > 0$ such that the integral system
\begin{align}
\gamma(t) &= \widetilde s_{p,N}(t)v_0 + \int_0^t \widetilde s_{p,N}(t-s)\left(\mathcal{Q}(V(\gamma,\sigma,\cdot),\gamma)(s) + \d_x \mathcal{R}(V(\gamma,\sigma,\cdot),\gamma,\gamma_t,\sigma_t)(s)\right. \nonumber\\
&\qquad \qquad  \left. + \,\d_x^2 \mathcal{P}(V(\gamma,\sigma,\cdot),\gamma)(s)\right)\de s \nonumber\\
\gamma_t(t) &= \partial_t \widetilde s_{p,N}(t)v_0 + \int_0^t \partial_t \widetilde s_{p,N}(t-s)\left(\mathcal{Q}(V(\gamma,\sigma,\cdot),\gamma)(s) + \d_x \mathcal{R}(V(\gamma,\sigma,\cdot),\gamma,\gamma_t,\sigma_t)(s)\right. \nonumber\\
&\qquad \qquad  \left. + \,\d_x^2 \mathcal{P}(V(\gamma,\sigma,\cdot),\gamma)\right)(s)\de s  \label{e:intsysgamma}\\
\sigma(t) &= \chi(t) \LA\widetilde{\Phi}_0,v_0\RA_{L^2_N} + \int_0^t \chi(t-s) \left(\LA\widetilde{\Phi}_0,\mathcal{Q}(V(\gamma,\sigma,\cdot),\gamma)(s)\RA_{L^2_N}\right. \nonumber\\
&\qquad \qquad \left. -\, \LA\widetilde{\Phi}_0',\mathcal{R}(V(\gamma,\sigma,\cdot),\gamma,\gamma_t,\sigma_t)(s)\RA_{L^2_N} + \LA\widetilde{\Phi}_0'',{\mathcal{P}}(V(\gamma,\sigma,\cdot),\gamma)(s)\RA_{L^2_N}\right) \de s,\nonumber\\
\sigma_t(t) &= \chi'(t) \LA\widetilde{\Phi}_0,v_0\RA_{L^2_N} + \int_0^t \chi'(t-s) \left(\LA\widetilde{\Phi}_0,\mathcal{Q}(V(\gamma,\sigma,\cdot),\gamma)(s)\RA_{L^2_N}\right.\nonumber\\
&\qquad \qquad \left. -\, \LA\widetilde{\Phi}_0',\mathcal{R}(V(\gamma,\sigma,\cdot),\gamma,\gamma_t,\sigma_t)(s)\RA_{L^2_N} + \LA\widetilde{\Phi}_0'',{\mathcal{P}}(V(\gamma,\sigma,\cdot),\gamma)(s)\RA_{L^2_N}\right) \de s,\nonumber
\end{align}
possesses a unique solution
\begin{align*} (\gamma,\gamma_t,\sigma,\sigma_t) \in C\big([0,\tau_{\max}),H^4_N \times H^2_N \times \R \times \R\big),\end{align*}
satisfying
\begin{align} \label{e:r0bound}
\left\|\left(\sigma(t),\sigma_t(t),\gamma(t),\gamma_t(t)\right)\right\|_{\R \times \R \times H^4_N \times H^2_N} < r_0, \qquad \|\gamma_x(t)\|_{L^\infty} \leq \frac12, 
\end{align}
for all $t \in [0,\tau_{\max})$. In addition, if $\tau_{\max} < T_{\max}$, then it holds
\begin{align} \label{e:gammablowup2}
\limsup_{t \uparrow \tau_{\max}} \left\|\left(\gamma,\gamma_t,\sigma,\sigma_t\right)\right\|_{H^4_N \times H^2_N \times \R \times \R} \geq r_0.
\end{align}
Finally, we have $(\gamma,\sigma) \in C^1\big([0,\tau_{\max}),H^2_N \times \R\big)$ with $\partial_t \left(\gamma,\sigma\right)(t) = \left(\gamma_t,\sigma_t\right)(t)$ for $t \in [0,\tau_{\max})$.
\end{proposition}
\begin{proof}
First, we note that, for any $j,l \in \mathbb N_0$ the operators $\partial_t^l \widetilde s_{p,N}(t)\partial_x^j \colon L^2_N \to H^4_N$ and $L^2_N \to \R, f \mapsto \smash{\partial_t^l \chi(t) \langle \partial_x^j \widetilde{\Phi}_0, f\rangle_{L^2_N}}$ are $t$- and $N$-uniformly bounded by Proposition~\ref{P:lin}. Second, the $N$-uniform embedding $H^1_N \hookrightarrow L^\infty(\R)$ yields an $N$-independent constant $r_0 > 0$ such that if $\gamma \in H^2_N$ satisfies $\smash{\|\gamma\|_{H^2_N}} \leq r_0$, then we have $\smash{\|\gamma_x\|_{L^\infty} \leq \tfrac12}$. Third, letting $B_1$ and $B_2$ be the closed balls of radius $r_0$ centered at the origin in $H_4^N \times \R$ and $H_4^N \times H^2_N \times \R \times \R$, respectively, it follows from Lemma~\ref{L:loc_gamma} that the nonlinear maps $\smash{\widetilde{\mathcal{Q}},\widetilde{\mathcal{P}}} \colon B_1 \times [0,T_{\max}) \to L^2_N$ and $\smash{\widetilde{\mathcal{R}}} \colon B_2 \times [0,T_{\max}) \to L^2_N$ given by
\begin{align*}
\widetilde{\mathcal{P}}(\gamma,\sigma,t) = \mathcal{P}\left(V\left(\gamma,\sigma,t\right),\gamma\right), \qquad
\widetilde{\mathcal{Q}}(\gamma,\sigma,t) = \mathcal{Q}\left(V\left(\gamma,\sigma,t\right),\gamma\right),
\end{align*}
and
\begin{align*}
\widetilde{\mathcal{R}}(\gamma,\gamma_t,\sigma,\sigma_t,t) = \mathcal{R}\left(V\left(\gamma,\sigma,t\right),\gamma,\gamma_t,\sigma_t\right),
\end{align*}
are well-defined, continuous in $t$ and locally Lipschitz continuous in $(\gamma,\sigma)$ and $(\gamma,\gamma_t,\sigma,\sigma_t)$, respectively (uniformly in $N$ and $t$ on compact subintervals of $[0,T_{\max})$). 

Abbreviating $A = (\gamma,\gamma_t,\sigma,\sigma_t)$, we thus find that the right-hand side of~\eqref{e:intsysgamma} is of the abstract form
$$\mathcal{S}_1(t)A_0 + \int_0^{t} \left(\mathcal{S}_1(t-s)\breve{\mathcal{N}}_1(A(s),s) + \mathcal{S}_2(t-s)\breve{\mathcal{N}}_2(A(s),s) + \mathcal{S}_3(t-s)\breve{\mathcal{N}}_3(A(s),s)\right) \de s,$$
with $A_0 = (v_0,v_0,v_0,v_0)$ and where $\mathcal{S}_i(t)$ are $t$- and $N$-uniformly bounded operators and the nonlinear maps $\smash{\breve{\mathcal{N}}}_i(A,t)$ are continuous in $t$ and locally Lipschitz continuous in $A$ (uniformly in $N$ and $t$ on compact subintervals of $[0,T_{\max})$). Thus, standard arguments, see for instance~\cite[Proposition 4.3.3]{CA98} or~\cite[Theorem 6.1.4]{Pazy}, now yield for any $\vartheta \in (0,r_0)$ a time $\tau_\vartheta > 0$ such that $\Psi \colon C\big([0,\tau_\vartheta],B_N(r_0 - \vartheta)\big) \to C\big([0,\tau_\vartheta],B_N(r_0 - \vartheta)\big)$, where $\Psi(A)[t]$ is given by the right-hand side of~\eqref{e:intsysgamma} and $B_N(r)$ is the closed ball centered at the origin in $H^4_N \times H^2_N \times \R \times \R$ of radius $r$, is a well-defined contraction mapping. Hence, by the Banach fixed point theorem, $\Psi$ admits a unique fixed point, which yields a unique solution $A \in C\big([0,\tau_\vartheta],B_N(r_0 - \vartheta)\big)$ of~\eqref{e:intsysgamma}. Letting $\tau_{\max} \in (0,T_{\max}]$ be the supremum of all such times $\tau_\vartheta$ with $\vartheta \in (0,r_0)$, we obtain a maximally defined solution $A \in C\big([0,\tau_{\max}),H^4_N \times H_2^N \times \R \times \R\big)$ of~\eqref{e:intsysgamma} obeying~\eqref{e:r0bound}.

Next, assume by contradiction that $\tau_{\max} < T_{\max}$ and~\eqref{e:gammablowup2} does not hold. Then, there exists a constant $\vartheta_0 \in (0,r_0/2)$ such that it holds $A(t) \leq r_0 - 2\vartheta_0$ for all $t \in [0,\tau_{\max})$. Take $t_0 \in [0,\tau_{\max})$. Similarly as before, one proves that there exists a constant $\delta > 0$, independent of $t_0$, such that $\Psi_{t_0} \colon C\big([t_0,\min\{t_0+\delta,T_{\max}\}],B_N(r_0 - \vartheta_0)\big) \to C\big([t_0,\min\{t_0+\delta,T_{\max}\}],B_N(r_0 - \vartheta_0)\big)$ given by
\begin{align*}
\Psi_{t_0}(\widetilde{A}) &= \mathcal{S}_1(t)A_0 + \int_0^{t_0} \left(\mathcal{S}_1(t-s)\breve{\mathcal{N}}_1(A(s),s) + \mathcal{S}_2(t-s)\breve{\mathcal{N}}_2(A(s),s\right) + \mathcal{S}_3(t-s)\breve{\mathcal{N}}_3(A(s),s) \de s \\
&\qquad + \, \int_{t_0}^{t} \left(\mathcal{S}_1(t-s)
\breve{\mathcal{N}}_1(\widetilde{A}(s),s) + \mathcal{S}_2(t-s)
\breve{\mathcal{N}}_2(\widetilde{A}(s),s) + \mathcal{S}_3(t-s)\breve{\mathcal{N}}_3(\widetilde{A}(s),s)\right)\de s,\end{align*}
is a well-defined contraction mapping, which admits a unique fixed point $\widetilde A \in C\big([t_0,\min\{t_0+\delta,T_{\max}\}], B_N(r_0-\vartheta_0)\big)$. Now setting $t_0 := \tau_{\max} - \delta/2$, it readily follows that the function $\check{A} \in C\big([0,\min\{\tau_{\max} + \delta/2,T_{\max}\}],B_N(r_0-\vartheta_0)\big)$ given by
\begin{align*}
 \check{A}(t) = \begin{cases} A(t), & t \in [0,\tau_{\max} - \frac{\delta}{2}], \\ \widetilde A(t), & t \in [\tau_{\max} - \frac{\delta}{2},\min\{\tau_{\max} + \frac{\delta}{2},T_{\max}\}],
                     \end{cases}
\end{align*}
solves~\eqref{e:intsysgamma}, which contradicts the maximality of $\tau_{\max}$. We conclude that, if $\tau_{\max} < T_{\max}$, then~\eqref{e:gammablowup2} must hold.

Finally, using Proposition~\ref{P:lin}, we observe that both $\gamma(t)$ and $\sigma(t)$ are differentiable on $[0,\tau_{\max})$ with $\partial_t (\gamma,\sigma)(t) = (\gamma_t,\sigma_t)(t)$, where we use $\widetilde s_{p,N}(0) = 0$ and $\chi(0) = 0$. This completes the proof.
\end{proof}

\bibliographystyle{abbrv}
\bibliography{LLE}

\end{document}